\documentclass[reqno]{amsart}
\usepackage{amssymb}
\usepackage[usenames, dvipsnames]{color}
\usepackage{enumitem}

\usepackage{mathrsfs}

\usepackage{marginnote}
\usepackage{todonotes}
\usepackage{hyperref}
\usepackage{esint}
\usepackage{verbatim}

\theoremstyle{plain}
\newtheorem{theorem}{Theorem}[section]
\newtheorem{lemma}[theorem]{Lemma}
\newtheorem{corollary}[theorem]{Corollary}
\newtheorem{proposition}[theorem]{Proposition}

\theoremstyle{definition}

\newtheorem{assumption}[theorem]{Assumption}

\theoremstyle{remark}
\newtheorem{remark}[theorem]{Remark}

\numberwithin{equation}{section}

\newcommand{\bR}{\mathbb{R}}
\newcommand{\bH}{\mathbb{H}}

\newcommand\cD{\mathcal{D}}

\newcommand\cH{\mathcal{H}}

\newcommand\cU{\mathcal{U}}

\newcommand\sW{\mathscr{W}}

\newcommand\sH{\mathscr{H}}

\makeatletter
\@namedef{subjclassname@2020}{%
  \textup{2020} Mathematics Subject Classification}
\makeatother

\begin{document}

\title[Dirichlet problem, singular or degenerate coefficients]{Parabolic and elliptic equations with singular or degenerate coefficients: the Dirichlet problem}

\author[H. Dong]{Hongjie Dong}
\address[H. Dong]{Division of Applied Mathematics, Brown University,
182 George Street, Providence, RI 02912, USA}
\email{Hongjie\_Dong@brown.edu}
\thanks{H. Dong is partially supported by the Simons Foundation, grant \# 709545.}

\author[T. Phan]{Tuoc Phan}
\address[T. Phan]{Department of Mathematics, University of Tennessee, 227 Ayres Hall,
1403 Circle Drive, Knoxville, TN 37996-1320, USA}
\email{phan@math.utk.edu}

\thanks{T. Phan is partially supported by the Simons Foundation, grant \# 354889.}

\begin{abstract} We consider the Dirichlet problem for a class of elliptic and parabolic equations in the upper-half space $\bR^d_+$, where the coefficients are the product of $x_d^\alpha, \alpha \in (-\infty, 1),$ and a bounded uniformly elliptic matrix of coefficients. Thus, the coefficients are singular or degenerate near the boundary $\{x_d =0\}$ and they may not locally integrable. The novelty of the work is that we find proper weights under which the existence, uniqueness, and regularity of solutions in Sobolev spaces are established. These results appear to be the first of their kind and are  new even if the coefficients are constant. They are also readily extended to systems of equations.
\end{abstract}


\maketitle

\section{Introduction}

Elliptic and parabolic equations with singular or degenerate coefficients appear quite naturally in both pure and applied problems.  As examples, we refer the reader to \cite{JW, Lin, Lin-Wang} for the problems in geometric PDEs, \cite{DHL, DHL-1} for problems from porous media, \cite{Athreya, Zh-Du} for problems in probability,  \cite{Heston, Pop-1} for problems arising in mathematical finance, \cite{Epstein, Epop-2} for problems in mathematical biology, and \cite{Cal-Syl, ST17} for problems related to fractional heat and fractional Laplace equations.  Due to these interests,  a lot of attention has been paid to regularity theory for equations with singular-degenerate coefficients. For examples, see the book \cite{OR} and the references therein for classical results, and also \cite{FKS, Lin, WWYZ}.

In this paper, we study the Dirichlet problem for a class of elliptic and parabolic equations in the upper-half space $\bR^d_+$ whose coefficients are singular or degenerate near the boundary $\partial \bR^d_+$ of the prototype $x_d^{\alpha}$ with $\alpha \in (-\infty, 1)$. We find a correct class of weighted Sobolev spaces for the existence, uniqueness, and regularity estimates of the solutions. This paper is also a companion of \cite{DP20}, in which the same type of equations but with the conormal boundary condition is considered when $\alpha \in (-1, \infty)$. A feature in these papers is that coefficients can be singular or degenerate near the boundary of the upper-half space in a way which may not satisfy the classical Muckenhoupt $A_2$ condition.  In fact, this work includes the supercritical case that the weight $x_d^\alpha$ is even not locally integrable near $\{x_d =0\}$.  As we point out in the paragraph after \eqref{0730.mu} below, in terms of  the $L_p$-theory, the natures of the boundary conditions for the Dirichlet problem considered here and the conormal problem in \cite{DP20} and also in \cite{DP19, D-P} are quite different. This is not the case for equations with bounded uniformly elliptic coefficients (i.e., $\alpha =0$) as seen in the classical theory.

Let  $\Omega_T =(-\infty,T)\times \bR^d_+$ be a space-time domain, where $T\in (-\infty,+\infty]$, $\bR_+= (0, \infty)$,  and $\bR^d_+ = \bR^{d-1} \times \bR_+$ is the upper-half space. Let $\alpha \in (-\infty,  1)$ be a fixed number and $\lambda>0$. Let $(a_{ij}): \Omega_T \rightarrow \bR^{d\times d}$ be a matrix of measurable coefficients, which satisfies the following ellipticity and boundedness conditions: there is a constant $\kappa \in (0,1)$ such that
\begin{equation} \label{ellipticity}
\begin{split}
\kappa |\xi|^ 2 \leq a_{ij}(t,x) \xi_i\xi_j \quad \text{and} \quad |a_{ij}(t,x)| \leq \kappa^{-1}
\end{split}
\end{equation}
for every $\xi = (\xi_1, \xi_2,\ldots, \xi_d) \in \bR^d$ and $(t,x)=(t,x',x_d) \in \Omega_T$.
In addition, we assume that $a_{ij}$ satisfy the partially small BMO condition. See Assumption \ref{assump1} below. We investigate the following  singular or degenerate parabolic equation
\begin{equation} \label{eq3.23}
x_d^\alpha ( u_t + \lambda  u ) - D_i\big(x_d^\alpha (a_{ij}(t,x) D_j u - F_i)\big)  = \sqrt{\lambda} x_d^\alpha f   \quad \text{in} \  \Omega_T
\end{equation}
with the zero Dirichlet boundary condition on the boundary $(-\infty, T) \times \partial \overline{\bR^{d}_+}$
\begin{equation} \label{m-bdr-cond}
u(t,x', 0)   = 0,  \quad (t,x') \ \in (-\infty, T) \times \bR^{d-1}
\end{equation}
as well as the corresponding elliptic equations. We note that by Remark \ref{bdr-remark} below, the boundary condition is automatically satisfied as long as the solution belongs to suitable function spaces when $\alpha \leq -1$.

Observe that by testing the equation \eqref{eq3.23} with $u$ and performing standard calculation, we have
\begin{equation}\label{ener-0728}
\begin{split}
& \int_{\Omega_T}|Du(z)|^2 \mu(dz) + \lambda \int_{\Omega_T} |u(z)|^2 \mu(dz)\\
& \leq N \int_{\Omega_T} \Big(|F(z)|^2 + |f(z)|^2 \Big)\mu(dz),
\end{split}
\end{equation}
where $\mu(dz) = x_d^\alpha dx dt$. Thus, it is tempting to prove an estimate like
\begin{equation} \label{0730.mu}
\begin{split}
& \int_{\Omega_T}|Du(z)|^p \mu(dz) + \lambda^{p/2} \int_{\Omega_T} |u(z)|^p \mu(dz)\\
& \leq N \int_{\Omega_T} \Big(|F(z)|^p + |f(z)|^p \Big) \mu(dz)
\end{split}
\end{equation}
using $\mu$ as the underlying measure. This is actually the case in \cite{DP19, D-P, DP20} for the conormal problems and in \cite{CMP} for elliptic equations in bounded domains with homogeneous Dirichlet boundary condition and with coefficients singular or degenerate as general $A_2$ weights satisfying a smallness condition on weighted mean oscillation. However, this does not seem to be a right approach here as the measure $\mu$ is not locally integrable near $\{x_d=0\}$ when $\alpha \leq -1$. In addition, Remark \ref{remark1} below indicates that \eqref{0730.mu} cannot be true when $p \geq \frac{1}{\alpha}+1$ even for $\alpha \in (0,1)$, $d=1$, and $(a_{ij})$ is an identity matrix. The novelty in our paper is an observation that \eqref{ener-0728} is equivalent to
\begin{align*}
& \int_{\Omega_T}|x_d^{\alpha}Du(z)|^2\mu_1(dz) + \lambda \int_{\Omega_T} |x_d^{\alpha} u(z)|^2\mu_1(dz) \\
& \leq N \int_{\Omega_T} \Big(|x_d^{\alpha} F(z)|^2 + |x_d^{\alpha}f(z)|^2 \Big) \mu_1(dz),
\end{align*}
where $\mu_1(dz) = x_d^{-\alpha} dx dt$. From this, we take $\mu_1$ as the underlying measure and establish the $L_p$-theory of \eqref{eq3.23} for $x_d^\alpha Du$ and $x_d^\alpha u$. This idea is clearly confirmed in the pointwise gradient estimates in Propositions \ref{interior-Linf}-\ref{lem1} and Corollary \ref{cor5.2} below, which is also a topic of independent interest. As a result, in Theorem \ref{thm2} we prove the following estimate
 \begin{equation} \label{08.20-eqn}
\begin{split}
& \int_{\Omega_T}|x_d^{\alpha}Du(z)|^p \mu_1(dz) + \lambda^{p/2} \int_{\Omega_T} |x_d^{\alpha} u(z)|^p\mu_1(dz) \\
& \leq N \int_{\Omega_T} \Big(|x_d^{\alpha} F(z)|^p + |x_d^{\alpha}f(z)|^p \Big) \mu_1(dz)
\end{split}
\end{equation}
for every weak solution $u$ of \eqref{eq3.23}-\eqref{m-bdr-cond} and  $p \in (1, \infty)$, which is quite different from the type of estimate \eqref{0730.mu} derived in \cite{DP19, D-P, DP20} and in \cite{CMP}. 

For local boundary estimates, when $u$ is a weak solution of $$
x_d^\alpha u_t- D_i\big(x_d^\alpha (a_{ij} D_j u - F_i)\big) =  x_d^\alpha f
$$
in the upper-half parabolic cylinder $Q_2^+$, one of our main results, Corollary \ref{main-thrm}, reads that
\begin{equation}
                                \label{eq6.35}
\begin{split}
&\Big(\int_{Q_1^+} \big(|u|^p+ |D u|^p\big) x_d^{(p-1)\alpha} \, dz\Big)^{1/p}
\leq N \int_{Q_2^+} \big(|u|+|Du|\big) \, dz \\
&\quad + N \Big(\int_{Q^+_2} |F|^p x_d^{(p-1)\alpha} \, dz\Big)^{1/p} +  N\Big(\int_{Q^+_2} |f|^{p^*} x_d^{(p^*-1)\alpha} \, dz\Big)^{1/p^*}
\end{split}
\end{equation}
for every $p\in (1,\infty)$, where $p^* \in [1, p)$ depending on $\alpha$, $p$, and $d$  as in \eqref{eq3.19}-\eqref{eq3-2.19} below and $N>0$ is a constant depending on $d$, $\alpha$, $p$, and $p^*$.  The elliptic version of \eqref{eq6.35} is also obtained. See Corollary \ref{ell-local-th}. To the best of our knowledge, \eqref{eq6.35} is new even in the elliptic case and when $a_{ij} =\delta_{ij}$.

Because the free terms $F$ and $f$ could be anisotropic due to the natures in applications, we also consider general $A_p$ weights with respect to the underlying measure $\mu_1$, and prove the existence, uniqueness, and estimates in mixed-norm weighted Sobolev spaces. In particular, for a specific power weight $x_d^\gamma$, we obtain the solvability and estimate in weighted and mixed-norm Sobolev spaces:
\begin{equation} \label{0805-1.eqn}
\begin{split}
& \int_{-\infty}^T\Big(\int_{\bR_+^d} \big(|u|^p+ |D u|^p\big) x_d^{\gamma} \, dx\Big)^{q/p}\,dt\\
& \leq N \int_{-\infty}^T\Big(\int_{\bR_+^d} \big(|F|^p+|f|^{p}\big) x_d^{\gamma} \, dx\Big)^{q/p}\,dt,
\end{split}
\end{equation}
where $p,q\in (1,\infty)$ and $\gamma$ is in the optimal range $(p\alpha -1, p-1)$. See Theorem \ref{thm1.4} and Remark \ref{rem2.6} for details.  A corresponding result for elliptic equations is also obtained. See Theorem \ref{ell-thm} below. Again, these results are new even when $(a_{ij})$ is the identity matrix.

It is worth noting that for the Laplace and heat equations, i.e., when $\alpha=0$ and $a_{ij} =\delta_{ij}$, a similar result as \eqref{0805-1.eqn} was first obtained by Krylov in \cite{Kr99} with exactly the same range of $\gamma$.  In fact, the necessity of such theory came from stochastic partial differential equations (SPDEs) and is well explained in \cite{Kr94}. See also subsequent work \cite{KKL, KK04,DK15} for deterministic equations and \cite{KL99,Kr09} for stochastic partial differential equations. In these papers, the parameter $\alpha$ is always assumed to be $0$.  In \cite{Sire-2}, the authors established H\"{o}lder and Schauder type estimates for scalar elliptic equations of a similar type under the conditions that the coefficient matrix is symmetric, sufficiently smooth, and the boundary is invariant with respect to $(a_{ij})$, i.e.,
\begin{equation*} %\label{sym-cond}
a_{dj} = a_{jd}=0, \quad j = 1,2,\ldots, d-1.
\end{equation*}
Compared to \cite{Sire-2}, we do not impose such invariance condition or the symmetry condition. In the technical level, our proofs of the pointwise gradient estimates in Propositions \ref{interior-Linf}, \ref{lem1}, and Corollary \ref{cor5.2} are based on energy estimates, Sobolev embedding theorems, and an iteration argument, while the proof in \cite{Sire-2} uses a special transformation to reduce the problem to an equation with the conormal boundary condition considered in \cite{Sire-1}, a blow-up argument, and a Liouville type theorem. We also refer to \cite{WWYZ} for other results about H\"older estimates for linear equations in the case when $a_{ij} = \delta_{ij}$ and $\alpha> 0$, and to \cite{Lin, Lin-Wang} for results about a class of degenerate quasilinear and fully nonlinear elliptic equations satisfying some monotonicity conditions. In these papers, the proofs are based on the comparison principle and constructions of barrier functions, and thus only work for scalar equations in non-divergence form.

It is worth noting that when $\alpha\in (-1,1)$, the weight function $x_d^\alpha$ is in the Muckenhoupt $A_2$ class. Interior H\"{o}lder continuity of weak solutions to elliptic equations with coefficients singular or degenerate with general $A_2$-weights was established long time ago in \cite{FKS}. The $W^{1}_p$-counterpart of this result was only obtained recently in \cite{CMP}. 

The study of the well-posedness of the Dirichlet problem for linear elliptic equations with coefficients degenerate near the boundary was initiated in the pioneering work \cite{Keldys}. In this work, an important observation similar to our Remark \ref{bdr-remark} below was made. Since then, there are many papers  devoted to the existence and uniqueness of $L_2$-weak solutions for this class of equations. See the discussion on pages 1-2 of the book \cite{OR} and the references therein. In particular,  in \cite{ichera-1, Fichera} the existence, uniqueness, and estimates in unweighted $L_p$ spaces for weak solutions were established when the coefficients are sufficiently smooth. Specifically, as stated in \cite[Eq. (4) and Theorem 1.2.1]{OR}, in these work the leading coefficients are assumed to be in $C^2$. In our case, the leading coefficients are $x_d^\alpha a_{ij}(t,x)$ and they are only measurable.  Therefore, the results and techniques in these mentioned work are not applicable to our case.

Let us give a brief description of the proofs. We first derive the weighted estimate and solvability for equations with coefficients depending only on $x_d$.
A crucial step in the proof is the pointwise gradient estimates for homogeneous equations, which is Proposition \ref{lem1}. For this, we exploit the idea in \cite{DP19} which uses energy estimates and the Sobolev embedding, and combine these with an elaborate iteration procedure. Because we do no use the maximum/comparison principle, the approach works also for systems of equations. For equations with partially small BMO coefficients, we apply the level set argument introduced in \cite{Cal}. To prove the weighted mixed-norm estimates, we further show a higher regularity result (see Corollary \ref{cor5.2}) and then apply the method of mean oscillation estimates introduced in \cite{Krylov} and developed in \cite{MR3812104}.

The remaining part of the paper is organized as follows. In Section \ref{sec2}, we introduce some notation and state the main results of the paper. Section \ref{sec3} is devoted to some preliminary results including the weighted Hardy inequality and weighted parabolic embedding as well as the unique solvability of the equation when $p=2$. In Section \ref{sec4}, we consider equations with coefficients depending only on $x_d$. Among other results, we prove interior and boundary pointwise gradient estimates for solutions to the homogeneous equations and a unique solvability result for any $p\in (1,\infty)$. Finally, the proofs of the main theorems are given in Section \ref{sec5}.

\section{Notation and main results}
            \label{sec2}
\subsection{Notation} Let $r>0$, $z_0 = (t_0, x_0)$ with $x_0 = (x_0', x_{0d}) \in  \bR^{d-1} \times \bR$ and $t_0 \in \bR$. We define $B_r(x_0)$  to be the ball in $\bR^d$ of radius $r$ centered at $x_0$, $Q_r(z_0)$ to be the parabolic cylinder of radius $r$ centered at $z_0$:
\[
Q_r(z_0) = (t_0-r^2 , t_0)\times B_r(x_0).
\]
Also, let $B_r^+(x_0)$ and $Q_r^+(z_0)$ be the upper-half ball and cylinder of radius $r$ centered at $x_0$ and $z_0$, respectively:
\begin{align*}
B_r^+(x_0) &= \big\{x = (x', x_d) \in \bR^{d-1} \times \bR:\, x_d >0, \ |x -x_0| < r\big\},\\
Q_r^+ (z_0)&= ( t_0-r^2, t_0)\times B_r^+ (x_0).
\end{align*}
When $x_0 =0$ and $t_0 =0$, for  simplicity of notation, we drop $x_0, z_0$ and write $B_r$, $B_r^+$, $Q_r$, and $Q_r^+$, etc.
We also define $B'(x_0')$ and $Q'(z'_0)$ to be the ball and the  parabolic cylinder in $\bR^{d-1}$ and $\bR^{d}$, where $z_0' = (t_0, x_0')$.

For a given non-negative Borel measure $\sigma$ on $\bR^{d+1}_+$ and for $p \in [1, \infty)$,  $-\infty\le S<T\le +\infty$, and $\cD \subset \bR^d_+$, let $L_p((S,T)\times \cD, d\sigma)$ be the weighted Lebesgue space consisting of measurable functions $u$ on $(S,T)\times \cD$ such that the norm
\[
\|u\|_{L_p( (S,T)\times \cD, d\sigma)}= \left( \int_{(S,T)\times \cD} |u(t,x)|^p\, d\sigma (t,x) \right)^{1/p} <\infty.
\]
%where $\mu(dz) = x_d^\alpha\,dxdt$.
For $p,q\in [1,\infty)$,  a non-negative Borel measure $\sigma$ on $\bR^{d}_+$, and the weights $\omega_0=\omega_0(t)$ and $\omega_1=\omega_1(x)$, we define $L_{q,p}((S,T)\times \cD,\omega\,d\sigma)$ to be the weighted and mixed-norm Lebesgue space on $\Omega_T$ equipped with the norm
\begin{equation*}
				%			\label{eq0806_02}
\|u\|_{L_{q,p}((S,T)\times \cD, \omega\, d\sigma)}=\left(\int_{S}^T\Big(\int_{\cD} |u(t,x)|^p \omega_1(x)\,\sigma(dx)\Big)^{q/p}\omega_0(t)\,dt\right)^{1/q},
\end{equation*}
where $\omega(t,x)=\omega_0(t)\omega_1(x)$.  We define the weighted Sobolev space
$$
W^1_p(\cD,\omega_1\,d\sigma)=\big\{u\in L_p(\cD,\omega_1\,d\mu):\,Du\in L_p(\cD,\omega_1\,d\sigma)\big\}
$$
equipped with the norm
$$
\|u\|_{W^1_p(\cD,\omega_1 d\sigma)}=\|u\|_{L_p(\cD,\omega_1 d\sigma)}+\|Du\|_{L_p(\cD,\omega_1d\sigma)}.
$$
The Sobolev space $\sW^1_p(\cD, \omega_1 d\sigma)$ is defined to be the closure in $W^1_p(\cD,\omega_1\,d\sigma)$ of all compactly supported functions in $C^\infty(\overline{\cD})$ vanishing near $\overline{\cD} \cap \{x_d=0\}$.

We also define
\[
\begin{split}
& \bH_{q, p}^{-1}( (S,T)\times \cD, \omega d\sigma) \\
& =\big\{u:\, u  =  D_iF_i+F_0/x_d +f\ \ \text{for some}\ f\in L_{q, p}( (S,T)\times \cD, \omega d\sigma)\\
& \qquad \ F= (F_0,\ldots,F_d) \in L_{q, p}((S,T)\times \cD, \omega d\sigma)^{d +1}\big\},
\end{split}
\]
which is equipped with the norm
\begin{align*}
\|u\|_{\bH_{q,p}^{-1}((S,T)\times \cD, \omega d\sigma)} &=\inf\big\{\|F\|_{L_{q,p}((S,T)\times \cD, \omega d\sigma)}
+\|f\|_{L_{q,p}((S,T)\times \cD, \omega d\sigma)}: \\
& \quad \qquad \ u= D_iF_i+F_0/x_d+f\big\}.
\end{align*}
Finally, we define the space
\[
\begin{split}
& \sH_{q,p}^1((S,T)\times \cD, \omega d\sigma)\\
& =\big\{u \in L_q(((S, T),\omega_0),  \sW^1_p(\cD, \omega_1 d\sigma)):    u_t\in  \bH_{q,p}^{-1}( (S,T)\times \cD, \omega d\sigma)\big\}
\end{split}
\]
equipped with the norm
\begin{align*}
&\|u\|_{\sH_{q,p}^1((S,T)\times \cD, \omega d\sigma)} \\
&= \|u\|_{L_{q,p}((S,T)\times \cD, \omega d\sigma)}
+ \|Du\|_{L_{q, p}((S,T)\times \cD,\omega d \sigma)}+\|u_t\|_{\bH_{q, p}^{-1}((S,T)\times \cD, \omega d\sigma)}.
\end{align*}
Alternatively, we can define $\sH_{q,p}^1((S,T)\times \cD, \omega d\sigma)$ to be the closure of all compactly supported functions in $C^\infty(\overline{(S,T)\times \cD})$ vanishing near $\overline{\cD \cap} \{x_d=0\}$ in the space
\begin{align*}
&\cH^1_{q,p}((S,T)\times \cD, \omega d\sigma)\\
& =\big\{u \in L_q(((S, T),\omega_0), W^1_p(\cD, \omega_1 d\sigma)):    u_t\in  \bH_{q,p}^{-1}( (S,T)\times \cD, \omega d\sigma)\big\},
\end{align*}
which is equipped with norm
\begin{align*}
&\|u\|_{\cH_{q,p}^1((S,T)\times \cD, \omega d\sigma)} \\
&= \|u\|_{L_{q,p}((S,T)\times \cD, \omega d\sigma)}
+ \|Du\|_{L_{q, p}((S,T)\times \cD,\omega d \sigma)}+\|u_t\|_{\bH_{q, p}^{-1}((S,T)\times \cD, \omega d\sigma)}.
\end{align*}
When $p =q$, we simply write $\sH_p^1(\Omega_T, \omega d\sigma) = \sH_{p,p}^1(\Omega_T, \omega d\sigma)$. Similar notation are also used for other spaces.

Throughout the paper, for $\alpha\in (-\infty,1)$, we define the measures
$$
d\mu(x)=\mu(dx)=x_d^\alpha\,dx,\quad d\mu_1(x)=\mu_1(dx)=x_d^{-\alpha}\,dx,
$$
and
$$
d\mu(z)=\mu(dz)=x_d^\alpha\,dxdt,\quad d\mu_1(z)=\mu_1(dz)=x_d^{-\alpha}\,dxdt.
$$
We say that $u \in \sH_{q,p}^1((S,T)\times \cD, \omega d\mu)$ is a weak solution to \eqref{eq3.23}-\eqref{m-bdr-cond} in $(S,T)\times \cD$ if
\begin{align*}
&\int_{(S,T)\times \cD} (-u \partial_t \varphi +\lambda u \varphi)\,\mu(dz) + \int_{(S,T)\times \cD}( a_{ij} D_{j} u  - F_i)D_{i} \varphi\,\mu(dz) \\
&= \lambda^{1/2} \int_{(S,T)\times \cD} f(z) \varphi(z)\,\mu(dz)
\end{align*}
for any $\varphi \in C_0^\infty((S,T)\times \cD)$.

We need the following assumption on the leading coefficients that was first introduced in \cite{MR2338417, MR2300337} in the unweighted case.
\begin{assumption}[$\delta_0, R_0$]
            \label{assump1}  For every $r\in (0, R_0]$ and $z_0=(z_0',x_d) \in \bR^d \times \overline{\bR_+}$, we have
\[
 \max_{i,j \in\{1,2,\ldots, d\}}\fint_{Q^+_{r}(z_0)} |a_{ij}(t,x) - [a_{ij}]_{r,z_0}(x_d)|\, \mu_1(dz)\\
\le \delta_0,
\]
where $[a_{ij}]_{r,z_0}(x_d)$ is the average of $a_{ij}$ with respect to $z'=(t, x')$  in the parabolic cylinder $Q'_{r}(z_0') \subset \bR^d$:
\begin{equation*} %\label{mean-a-0724}
[a_{ij}]_{r,z_0}(x_d) = \fint_{Q'_{r}(z_0')}a_{ij}(t,x', x_d) \,dx'\, dt, \quad  i, j \in \{1, 2,\ldots, d\}.
\end{equation*}
\end{assumption}
\noindent
In the above assumption and throughout the paper, for a measurable set $\Omega\subset \bR^{d+1}$ and any integrable function $f$ on $\Omega$ with respect to some locally finite Borel measure $\sigma$, we write
$$
\fint_\Omega f(z)\ \sigma(dz) =\frac 1 {\sigma(\Omega)}\int_\Omega f(z)\, \sigma(dz), \quad \text{where} \quad \sigma(\Omega) = \int_{\Omega} \sigma(dz).
$$
\noindent
Assumption \ref{assump1} means that the mean oscillations of $(a_{ij})$ with respect to the $z'$-variable in parabolic cylinders of radius at most $R_0$ are smaller than $\delta_0$. Therefore, if the matrix $(a_{ij})$ only depends on the $x_d$-variable, the assumption is always satisfied.

\subsection{Main theorems} Our first main result is about the existence, uniqueness, and global regularity estimates of solutions to the divergence form equation \eqref{eq3.23}.
\begin{theorem} \label{thm2}
Let $\alpha \in (-\infty, 1)$, $\kappa \in (0,1)$, $R_0 \in (0, \infty)$, and $p \in (1,\infty)$.  Then there exist $\delta_0=\delta_0(d,\kappa,\alpha,p) \in (0,1)$ and $\lambda_0~=~\lambda_0(d,\kappa,\alpha, p)~\ge~0$ such that the following assertions hold. Suppose that \eqref{ellipticity}  and \textup{Assumption \ref{assump1} ($\delta_0, R_0$)} are satisfied. If $u\in \sH^1_p(\Omega_T, x_{d}^{\alpha p} d\mu_1)$ is a weak solution of
\eqref{eq3.23}-\eqref{m-bdr-cond} for some $\lambda~\ge~\lambda_0 R_0^{-2}$, $f\in L_p(\Omega_T, x_d^{\alpha p}d\mu_1)$, and $F \in L_p(\Omega_T, x_d^{\alpha p} d\mu_1)^d$,  then we have
\begin{equation} \label{main-thm-est}
\begin{split}
& \| D u\|_{L_p(\Omega_T, x_d^{\alpha p} d\mu_1)}+\sqrt{\lambda} \|u\|_{L_p(\Omega_T,x_d^{\alpha p} d\mu_1)}\\
& \leq N \|F\|_{L_p(\Omega_T, x_d^{\alpha p} d\mu_1)} + N\| f\|_{L_p(\Omega_T,x_d^{\alpha p} d\mu_1)},
\end{split}
\end{equation}
where $N=N(d,\kappa,\alpha,p)~>~0$. Moreover, for any $\lambda~>~\lambda_0 R_0^{-2}$, $f\in L_p(\Omega_T,x_d^{\alpha p} d\mu_1)$, and $F \in L_p(\Omega_T, x_d^{\alpha p}d\mu_1)^d$, there exists a unique weak solution $u\in \sH^1_p(\Omega_T, x_d^{\alpha p}d\mu_1)$ to \eqref{eq3.23}-\eqref{m-bdr-cond}.
\end{theorem}

In the next result, we give a local boundary estimate in a upper-half cylinder. Consider
 \begin{equation} \label{eqn-entire}
\left\{
\begin{aligned}
x_d^\alpha u_t- D_i\big(x_d^\alpha (a_{ij} D_j u - F_i)\big)   =  x_d^\alpha f  \quad &\text{in} \ \ Q_2^+\\
u  =  0 \quad &\text{on} \ \partial Q_2^+ \cap \{x_d =0\}.
\end{aligned}\right.
\end{equation}
Let $p  \in [1,\infty)$ and $p^*\in [1,p)$ satisfy
\begin{equation}
                                \label{eq3.19}
\left\{\begin{aligned}
\frac{1}{p^*} \le \frac{1}{d+2+ \alpha_-}+\frac{1}{p} &\quad\text{when}\ p^*>1\\
\frac{1}{p^*}< \frac{1}{d+2+ \alpha_-}+\frac{1}{p} &\quad\text{when}\ p^*=1,
\end{aligned}\right.
\end{equation}
if $d \geq 2$ or $\alpha=0$, and
\begin{equation}
                                \label{eq3-2.19}
\left\{\begin{aligned}
\frac{1}{p^*}\le \frac{1}{4+ \alpha_-}+\frac{1}{p}&\quad\text{when}\ p^*>1\\
\frac{1}{p^*} < \frac{1}{4+ \alpha_-}+\frac{1}{p}&\quad\text{when}\ p^*=1,
\end{aligned}\right.
\end{equation}
if $d =1$  and $\alpha\neq 0$, where $\alpha_- = \max\{-\alpha, 0\}$.

\begin{corollary} \label{main-thrm}
Let $\alpha \in (-\infty, 1)$, $\kappa \in (0,1)$, $R_0\in (0,\infty)$,  $1<p_0 \le p<\infty$, and $p^*\in [1,p)$ satisfy \eqref{eq3.19}-\eqref{eq3-2.19}. Then there exists $\delta_0=\delta_0(d,\kappa,\alpha,p_0,p) \in (0,1)$ such that the following assertion holds.  Suppose that \eqref{ellipticity} and  \textup{Assumption} \ref{assump1} \textup{(}$\delta_0, R_0$\textup{)} are satisfied. If $u\in \sH^1_{p_0}(Q_2^+, x_{d}^{\alpha p_0}d\mu_1)$ is a weak solution of \eqref{eqn-entire}, $F\in L_p(Q_2^+, x_{d}^{\alpha p}d\mu_1)^{d}$, and $f\in L_{p^*}(Q_2^+, x_d^{\alpha p^*}d\mu_1)$, then $u \in \sH_p^1(Q_1^+, x_{d}^{\alpha p}d\mu_1)$ and
\begin{align} \label{main-thm-estb}
& \| u\|_{L_p(Q_{1}^+, x_d^{\alpha p} d\mu_1)}+ \|D u\|_{L_p(Q_{1}^+, x_d^{\alpha p} d\mu_1)} \nonumber\\
&\leq  N \|u\|_{L_1(Q_2^+)} + N \|Du\|_{L_1(Q_2^+)} \\ \nonumber
& \qquad + N\| F\|_{L_p(Q_2^+, x_d^{\alpha p} d\mu_1)} + N\| f\|_{L_{p^*}(Q_2^+, x_d^{\alpha p^*} d\mu_1)},
\end{align}
where $N=N(d,\kappa,\alpha, p, p^*, R_0)>0$.
\end{corollary}

The conditions of $p$ and $p^*$ in \eqref{eq3.19} and \eqref{eq3-2.19} are due the Sobolev embedding result, Lemma \ref{lem2.2} below. When $d \geq 2$, the condition \eqref{eq3.19} of $p$ and $p^*$ is optimal. However, due to some technical difficulty, we could not reach the optimal condition when $d=1$.  We expect that when $d =1$, Corollary \ref{main-thrm} still holds when $p$ and $p^*$ satisfy the same condition as \eqref{eq3.19}.

The following remark confirms that regularity estimate using the measure $\mu$ as in \eqref{0730.mu}  may not valid even for $\alpha \in (0,1)$, $d=1$, and $a_{ij} =\delta_{ij}$.
\begin{remark}
        \label{remark1}
For $\alpha \in (0, 1)$, let $u(x) = x^{1-\alpha}$ for $x \in [0,\infty)$. We see that $u \in \sW^{1,2}((0,1), d\mu)$ is a weak solution of
\[
(x^\alpha u'(x))' = 0, \quad x  \in (0,1) \quad \text{and} \quad u(0) = 0.
\]
But
\[
\int_0^1 |u'(x)|^p x^{\alpha}dx = N \int_0^1 x^{\alpha(1-p)} dx < \infty
\]
only if $p < \frac{1}{\alpha} +1$.
\end{remark}

Our last result about \eqref{eq3.23}-\eqref{m-bdr-cond} is about the estimate and solvability in weighted and mixed-norm Sobolev spaces. For $p \in (1, \infty)$, a locally integrable function $\omega : \bR^d_+ \rightarrow \bR_+$ is said to be in $A_p(\bR^d_+, \mu_1)$ Muckenhoupt class of weights if \begin{equation*}
\begin{split}
& [\omega]_{A_p(\bR^d_+, \mu_1)}\\
& =
\sup_{r >0,x \in \overline{\bR^d_+}} \Big(\fint_{B^+_r(x)} \omega(y)\, \mu_1(dy) \Big)\Big(\fint_{B^+_r(x)} \omega(y)^{\frac{1}{1-p}}\, \mu_1(dy) \Big)^{p-1}<\infty.
\end{split}
\end{equation*}
Similarly, a locally integrable function $\omega : \bR \rightarrow \bR_+$ is said to be in $A_p(\bR)$ Muckenhoupt class of weights if
\begin{equation*}
[\omega]_{A_p(\bR)} =
\sup_{r >0,t \in \bR}
\Big(\fint_{t-r^2}^{t+r^2} \omega(s)\, ds \Big)\Big(\fint_{t-r^2}^{t+r^2} \omega(s)^{\frac{1}{1-p}}\, ds \Big)^{p-1}<\infty.
\end{equation*}
%\todo{I re-arranged the statement a little bit}
\begin{theorem} \label{thm1.4}
Let $\alpha \in (-\infty, 1)$, $\kappa \in (0,1)$,  $R_0 \in (0, \infty)$, $p, q, K \in (1,\infty)$. Then there exist
$$
\delta_0=\delta_0(d, \kappa, \alpha, p, q,  K)\in (0,1) \quad \text{and}\quad
\lambda_0=\lambda_0(d, \kappa, \alpha, p, q, K)\ge 0,
$$
such that the following assertions hold.  Suppose that \eqref{ellipticity} and \textup{Assumption \ref{assump1} ($\delta_0, R_0$)} are satisfied, and $\omega(t,x) =\omega_0(t)\omega_1(x)$  with $\omega_0\in A_q(\bR)$, $\omega_1\in A_p(\bR^d_+, \mu_1)$ such that
$$
 [\omega_0]_{A_q(\bR)}\le K,\quad [\omega_1]_{A_p(\bR^d_+, \mu_1)}\le K.
$$
If $u\in {\sH}^1_{q,p}(\Omega_T,\omega x_d^{\alpha p}d\mu_1)$ is a weak solution of
\eqref{eq3.23}-\eqref{m-bdr-cond} for some $\lambda~\ge~\lambda_0 R_0^{-2}$, $f\in L_{q,p}(\Omega_T,\omega x_d^{\alpha p}d\mu_1)$, and $F \in L_{q,p}(\Omega_T,\omega x_d^{\alpha p}d\mu_1)^d$, then we have
\begin{align} \label{main-thm-estc}
&\|D u\|_{L_{q,p}(\Omega_T,\omega x_d^{\alpha p}d\mu_1)}+\sqrt{\lambda} \|u\|_{L_{q,p}(\Omega_T,\omega x_d^{\alpha p}d\mu_1)}\nonumber\\
&\leq N \|F\|_{L_{q,p}(\Omega_T,\omega x_d^{\alpha p}d\mu_1)} + N\|f\|_{L_{q,p}(\Omega_T,\omega x_d^{\alpha p}d\mu_1)},
\end{align}
where $N=N(d, \kappa, \alpha, p, q,  K)>0$.
Moreover, for any $\lambda > \lambda_0 R_0^{-2}$,
$$
f\in L_{q, p}(\Omega_T, \omega x_d^{\alpha p}d\mu_1),\quad \text{and}\quad
F \in L_{q,p}(\Omega_T, \omega x_d^{\alpha p}d\mu_1)^d,
$$
there exists a unique weak solution  $u\in {\sH}^1_{q, p}(\Omega_T, \omega x_d^{\alpha p}d\mu_1)$ to \eqref{eq3.23}-\eqref{m-bdr-cond}.
\end{theorem}
\begin{remark}\label{rem2.6} Let us consider the special case when $\omega_{1}(x_d)=x_d^\beta$. It is easily seen that $\omega_{1} \in A_p(\bR^{d}_+, \mu_1)$ if and only if $\beta \in (\alpha -1, (1-\alpha) (p-1))$. Therefore, from Theorem \ref{thm1.4} we obtained the estimate and solvability in the space ${\sH}^1_{q,p}(\Omega_T,x_d^{\gamma} dz)$, where $\gamma=\beta+\alpha p-\alpha\in (p\alpha-1,p-1)$.
\end{remark}

Remark \ref{rem2.6} implies the following result which is a generalization of Corollary \ref{main-thrm}.
\begin{remark}
            \label{rem2.7}
Let $\alpha \in (-\infty, 1)$, $\kappa \in (0,1)$, $R_0\in (0,\infty)$,  $1<p_0 \le p<\infty$,  $p^*\in [1,p)$, and $\gamma\in (p\alpha-1,p-1)$. Assume that \eqref{eq3.19}-\eqref{eq3-2.19} is satisfied with $\tilde\alpha=\gamma/(p-1)$ in place of $\alpha$. Then there exists $\delta_0=\delta_0(d,\kappa,\alpha, \gamma, p_0,p) \in (0,1)$ such that the following assertion holds.  Suppose that \eqref{ellipticity} and  \textup{Assumption} \ref{assump1} \textup{(}$\delta_0, R_0$\textup{)} are satisfied. If $u\in \sH^1_{p_0}(Q_2^+, x_{d}^{\gamma+ \tilde{\alpha} (p_0-p)}dz)$ is a weak solution of \eqref{eqn-entire}, $F\in L_p(Q_2^+, x_{d}^{\gamma}dz)^{d}$, and $f\in L_{p^*}(Q_2^+, x_d^{\tilde\alpha(p^*-1)}dz)$, then $u \in \sH_p^1(Q_1^+, x_{d}^{\gamma}dz)$ and
\begin{align}  \label{remark-estz}
& \| u\|_{L_p(Q_{1}^+, x_d^{\gamma} dz)}
+\|D u\|_{L_p(Q_{1}^+, x_d^{\gamma} dz)} \nonumber\\
&\leq  N \|u\|_{L_1(Q_2^+)} + N \|Du\|_{L_1(Q_2^+)} \\ \nonumber
& \qquad + N\| F\|_{L_p(Q_2^+, x_d^{\gamma} dz)} + N\| f\|_{L_{p^*}(Q_2^+, x_d^{\tilde\alpha(p^*-1)}dz)},
\end{align}
where $N=N(d,\kappa,\alpha, \gamma, p, p^*, R_0)>0$.\end{remark}
\noindent
Though the proof of the above result is similar to that of Corollary \ref{main-thrm},  we  provide it in Appendix \ref{App-C} for completeness. Note also that it is possible to extend the result to equations with unbounded lower-order coefficients as in \cite{DK15}. However, we choose not to include it here for simplicity.

We now consider the elliptic equation
\begin{equation} \label{ell-eqn}
\left\{
\begin{aligned}
-D_i (x_d^\alpha [a_{ij}(x) D_j u - F_i]) + \lambda x_d^\alpha u = \sqrt{\lambda} x_d^\alpha  f(x) \quad &\text{in} \ \bR^d_+\\
u  = 0  \quad &\text{on} \ \partial \bR^d_+,
\end{aligned} \right.
\end{equation}
where $(a_{ij}): \bR_+^d \rightarrow \bR^{d\times d}$, $F: \bR^d_+ \rightarrow \bR^d$, and $f: \bR_+^d \rightarrow \bR$ are independent of $t$. Note  that \eqref{ellipticity}  and Assumption \ref{assump1} can be stated similarly in the time-independent case. Also, for each weight $\omega: \bR^d_+ \rightarrow \bR_+$ and for $p \in (1, \infty)$, a function $u \in \sW^{1}_{p}(\bR^d_+, \omega d\mu_1)$ is said to be a weak solution of \eqref{ell-eqn} if
\[
\int_{\bR^d_+} x_d^\alpha [a_{ij}D_j u - F_i] D_i \varphi(x)\, dx + \lambda \int_{\bR^d_+} x_d^\alpha u(x) \varphi(x)\, dx= \sqrt{\lambda} \int_{\bR^d_+} x_d^\alpha f(x) \varphi(x)\, dx
\]
for all $\varphi \in C_0^\infty(\bR^d_+)$.
%===
\begin{theorem} \label{ell-thm} Let $\alpha \in (-\infty, 1)$, $\kappa \in (0,1)$,  $R_0 \in (0, \infty)$, $p, K \in (1,\infty)$. Then there exist
$$
\delta_0=\delta_0(d, \kappa, \alpha, p, q,  K)\in (0,1) \quad \text{and}\quad
\lambda_0=\lambda_0(d, \kappa, \alpha, p, q, K)\ge 0,
$$
such that the following statements hold. Suppose that \eqref{ellipticity} and \textup{Assumption \ref{assump1}~($\delta_0, R_0$)} are satisfied, and $\omega \in A_p(\bR^d_+, \mu_1)$ such that $[\omega]_{A_p(\bR^d_+, \mu_1)}\le K$. If $u\in \sW^{1}_{p}(\bR^d_+,\omega x_d^{\alpha p}d\mu_1)$ is a weak solution of
\eqref{ell-eqn} for some $\lambda~\ge~\lambda_0 R_0^{-2}$, $f\in L_{p}(\bR^d_+,\omega x_d^{\alpha p}d\mu_1)$, and $F \in L_{p}(\bR^d_+,\omega x_d^{\alpha p}d\mu_1)^d$, then we have
\begin{align*}
&\|D u\|_{L_{p}(\bR^d_+,\omega x_d^{\alpha p}d\mu_1)}+\sqrt{\lambda} \|u\|_{L_{p}(\bR^d_+,\omega x_d^{\alpha p}d\mu_1)}\nonumber\\
&\leq N \|F\|_{L_{p}(\bR^d_+,\omega x_d^{\alpha p}d\mu_1)} + N\|f\|_{L_{p}(\bR^d_+,\omega x_d^{\alpha p}d\mu_1)},
\end{align*}
where $N=N(d, \kappa, \alpha, p, q,  K)>0$.
Moreover, for any $\lambda > \lambda_0 R_0^{-2}$,
$$
f\in L_{p}(\bR^d_+, \omega x_d^{\alpha p}d\mu_1), \quad \text{and}\quad
F \in L_{p}(\bR^d_+, \omega x_d^{\alpha p}d\mu_1)^d,
$$
there exists a unique weak solution  $u\in \sW^{1}_{p}(\bR^d_+,  \omega x_d^{\alpha p}d\mu_1)$ to \eqref{ell-eqn}.
\end{theorem}
Theorem \ref{ell-thm} can be derived from Theorem \ref{thm1.4} by viewing solutions to elliptic equations as steady state solutions of the corresponding parabolic equations. See, for example, the proofs of \cite[Theorem 2.6]{Krylov} or  \cite[Theorem 1.2]{D-P} for details. We therefore omit the proof.
%===

We are also interested in the local regularity estimates for elliptic equations. Consider the equation
\begin{equation} \label{eq.0612}
\left\{
\begin{aligned}
- D_i(x_d^\alpha[ a_{ij}(x) D_j u - F_i]) = x_d^\alpha f  \quad &\text{in} \ B_2^+\\
 u =  0  \quad  &\text{on} \ B_2 \cap \{x_d =0\}.
\end{aligned} \right.
\end{equation}
For a given weight $\omega: B_2^+ \rightarrow \bR_+$, a function $u \in \sW^{1}_{p}(B_2^+, \omega d\mu)$ is said to be a weak solution of \eqref{eq.0612} if
\[
\int_{B_2^+} x_d^\alpha \big( a_{ij} D_j u - F_i  \big) D_i \varphi dx = \int_{B_2^+} x_d^\alpha f(x) \varphi(x) dx, \ \forall \ \varphi \in C_0^\infty(B_2^+).
\]
%====
For each $p  \in (1,\infty), \tilde{\alpha} \in (-\infty, 1)$, let $\hat{p} \in [1,p)$ satisfy% \todo{I reformatted the expression}
\begin{equation} \label{eq1.0612}
\left\{                               \begin{aligned}
\frac{1}{\hat{p}} \le \frac{1}{d+\tilde{\alpha}_{-}}+ \frac{1}{p}\quad&\text{when}\ \hat{p}>1,\\
\frac{1}{\hat{p}} < \frac{1}{d+ \tilde{\alpha}_{-}}+\frac{1}{p} \quad&\text{when}\ \hat{p}=1,
\end{aligned}\right.
\end{equation}
where $\tilde{\alpha}_{-} = \max\{-\tilde{\alpha}, 0\}$. Our local regularity result for elliptic equations is the following corollary. %\todo{I upgraded this result}
\begin{corollary}
                \label{ell-local-th}
Let $\alpha \in (-\infty, 1)$, $\kappa \in (0,1)$, $R_0\in (0,\infty)$, $1<p_0 \le p<\infty$, and $\gamma \in (p\alpha -1, p-1)$.  Then there exists $\delta_0=\delta_0(d,\kappa,\alpha, \gamma, p_0, p) \in (0,1)$ such that the following assertion holds.  Suppose that $(a_{ij})$ satisfies \eqref{ellipticity} and  \textup{Assumption} \ref{assump1}~\textup{(}$\delta_0, R_0$\textup{)}.  For  $\hat{p} \in [1,p)$ and $\tilde{\alpha} = \gamma/(p-1)$ satisfying \eqref{eq1.0612}, if $u\in \sW^{1}_{p_0}(B_2^+, x_{d}^{\gamma +\tilde{\alpha}(p_0-p)}dx)$ is a weak solution of \eqref{eq.0612}, $F\in L_p(B_2^+, x_d^{\gamma} dx)^{d}$, and $f\in L_{\hat{p}}(B_2^+, x_d^{\tilde{\alpha}(\hat{p}-1)}dx)$, then $u \in \sW^1_{p}(B_1^+, x_d^{\gamma}dx)$ and
\begin{align*}% \label{main-thm-estbb}
  \|u\|_{\sW^1_p(B_{1}^+, x_d^{\gamma}dx)} & \leq  N \|u\|_{W^1_1(B_2^+, dx)} + N\|F\|_{L_p(B_2^+, x_d^{\gamma}dx)} \\ \notag
  & \quad + N\|f\|_{L_{\hat{p}}(B_2^+, x_d^{\tilde{\alpha}(\hat{ p} -1)}dx)},
\end{align*}
where $N=N(d,\kappa,\alpha, p, \hat{p}, R_0)>0$.
\end{corollary}
The proof of  Corollary \ref{ell-local-th} is sketched in Appendix \ref{App-C}. Note that the condition \eqref{eq1.0612} is due to the corresponding weighted embedding inequality in elliptic case (see Remark \ref{imbd-remark} and Lemma \ref{imbed-elli}), which is optimal.
%=====

\begin{remark} From the proofs below, we can see that all of the above results can be extended to systems of equations satisfying the strong ellipticity condition on coefficients. We stated them for scalar equations only for simplicity.
\end{remark}

\section{Sobolev spaces and  \texorpdfstring{$L_2$}{L2}-solutions}
                \label{sec3}
\begin{lemma}[Hardy's inequality]
                    \label{lem3.1}
For any $p\in [1,\infty)$, $\alpha\in (-\infty,p-1)$, and $u\in \sW^{1}_{p}(B_1^+,d\mu)$, we have
$$
\|u/x_d\|_{L_p(B_1^+,d\mu)}\le N\|D_du\|_{L_p(B_1^+,d\mu)},
$$
where $N=N(p,\alpha)>0$ is a constant.
\end{lemma}
\begin{proof}
We write
\begin{equation}
                                        \label{eq2.37}
u(x',x_d)/x_d=\int_0^1 (D_du)(x',sx_d)\,ds.
\end{equation}
By the Minkowski inequality and a change of variables,
\begin{align*}
&\|u/x_d\|_{L_p(B_1^+,d\mu)}\le \int_0^1 \|D_du(\cdot,s\cdot)\|_{L_p(B_1^+,d\mu)}\,ds\\
&\le \|D_du\|_{L_p(B_1^+,d\mu)}\int_0^1 s^{-(\alpha+1)/p}\,ds\le N\|D_du\|_{L_p(B_1^+,d\mu)},
\end{align*}
where we used $\alpha<p-1$ in the last inequality.
\end{proof}

\begin{lemma}
            \label{trace-u}
Let $\alpha  \le-1$ and $p \in [1, \infty)$. Then for any $u \in  W^1_p(B_1^+,d\mu)$, we have $u(x', 0) =0$ in sense of trace for a.e. $x' \in B_1'$. Moreover, $W^1_p(B_1^+,d\mu)=\sW^1_p(B_1^+,d\mu)$.
\end{lemma}
\begin{proof} For any $(x',x_d)\in B_1^+$, by the fundamental theorem of calculus and H\"{o}lder's inequality, we have for any $y\in (0,x_d)$,
\[
\begin{split}
|u(x', x_d)| & \leq |u(x', y)| + \int_{y}^{x_d} |D_d u(x', s)|\, ds  \\
&\leq  |u(x', y)| + N x_d^{1-\frac{1+\alpha}{p}}\left( \int_0^{x_d} |D_d u(x', s)|^ps^\alpha\, ds \right)^{1/p}.
\end{split}
\]
Integrating the above inequality with respect to $y\in (0,x_d)$ gives
\[
\begin{split}
 |u(x', x_d)| x_d
&\leq N \int_0^{x_d} \Big[ |u(x', y)|  +  x_d^{1-\frac{1+\alpha}{p}}\left( \int_0^{x_d} |D_d u(x', s)|^ps^\alpha \,ds \right)^{1/p} \Big]\, dy \\
& \leq N x_d^{1-\frac{1+\alpha}{p}}
\Big(\int_0^{x_d} |u(x', y)|^p y^{\alpha}\, dy \Big)^{1/p}\\
&\quad +Nx_d^{2-\frac{1+\alpha}{p}}\Big(\int_0^{x_d} |D_du(x', y)|^p y^{\alpha}\, dy \Big)^{1/p},
\end{split}
\]
where we used H\"{o}lder's inequality in the last  inequality to control the first term on the right-hand side. It then follows that
\[
\begin{split}
 |u(x', x_d)|  & \leq N x_d^{-\frac{1+ \alpha}{p}}\left(\int_0^{x_d} |u(x', y)|^p y^{\alpha}\, dy \right)^{1/p}  \\
 & \quad + N x_d^{1-\frac{1+\alpha}{p}}\left( \int_0^{x_d} |D_d u(x', y)|^p y^\alpha\, dy \right)^{1/p}.
\end{split}
\]
We take the $L_p$ norm of both sides of the above inequality in $x' \in B_{(1-x_d^2)^{1/2}}'$. As both the powers $-\frac{1+\alpha}{p}$ and $1-\frac{1+\alpha}{p}$ are nonnegative, by sending $x_d \rightarrow 0^+$, we obtain
\begin{equation}  \label{eq2.45}
              u(x', 0) = 0\quad \text{for} \ x' \in B_1'.
\end{equation}
This proves the first assertion of the lemma.

For the second assertion, it suffices to show that there is a sequence of functions $u_k\in W^1_p(B_1^+,d\mu)$, which vanish near $\{x_d=0\}$ and converge to $u$. We take a smooth function $\eta=\eta(x_d)$ such that $\eta=0$ when $x_d\le 0$ and $\eta=1$ when $x_d\ge 1$. For $k=1,2,\ldots$, we define $u_k=u\eta(kx_d)$. By the dominated convergence theorem, it is easily seen that $u_k\to u$ in $L_{p}(B_1^+,d\mu)$ as $k \rightarrow \infty$. Moreover, for $j=1,\ldots,d$,
$$
D_ju_k=\eta(kx_d) D_j u+u\delta_{dj}k\eta'(kx_d),
$$
where $\delta_{dj} =1$ if and only if $j =d$. Again by the dominated convergence theorem, we have
$$
\eta(kx_d) D_j u  \to D_j u\quad \text{in}\ L_{p}(B_1^+,d\mu).
$$
Since  %\todo{replace ${\bf 1}$ by $\chi$ for consistence}
$$
|k\eta'(kx_d)|\le Nx_d^{-1}\chi_{(0,1/k)}(x_d),
$$
by using \eqref{eq2.45} and \eqref{eq2.37}, we get
$$
|uk\eta'(kx_d)|\le N|u|x_d^{-1}\chi_{(0,1/k)}(x_d)
\le \chi_{(0,1/k)}(x_d)\int_0^1 |D_du(x',sx_d)|\,ds.
$$
Now similar to the proof of Lemma \ref{lem3.1}, we have
$$
\|uk\eta'(kx_d)\|_{L_p(B_1^+,d\mu)}\le N\|D_du\|_{L_p(B_1^+\cap\{x_d\in (0,1/k)\},d\mu)}\to 0
$$
as $k\to \infty$. Therefore, we conclude that $u_k\to u$ in $W^1_p(B_1^+,d\mu)$. The lemma is proved.
\end{proof}

\begin{remark} \label{bdr-remark} As a result of Lemma \ref{trace-u}, we only need to impose the boundary condition for the solution of \eqref{eq3.23} when $\alpha \in (-1,1)$ as long as it is in $W_p^{1}(\bR^d_+, d\mu)$.
\end{remark}
%===
\begin{lemma}[Weighted parabolic embedding]
                            \label{lem2.2}
Let $\alpha\in (-\infty, 1)$ and $q,q^*\in (1,\infty)$ satisfy
\begin{equation}
                        \label{eq2.06}
\left\{
\begin{aligned}
\frac{1}{q} \le  \frac{1}{d+2+ \alpha_-}+ \frac{1}{q^*} & \quad \text{if} \quad d \geq 2\\
\frac{1}{q}  \le  \frac{1}{4+ \alpha_-}+ \frac{1}{q^*} & \quad \text{if} \quad d =1.
\end{aligned} \right.
\end{equation}
Then for any $u\in \sH^1_q(Q_2^+, x_{d}^{\alpha q}d\mu_1)$, we have
\begin{equation}
                    \label{eq4.15bb}
\|x_{d}^{\alpha} u\|_{L_{q^*}(Q_2^+,d\mu_1)}\le N\|u\|_{\sH^1_{q}(Q_2^+, x_{d}^{\alpha q}d\mu_1)},
\end{equation}
where $N=N(d,\alpha, q, q^*)>0$ is a constant and $\alpha_- =\max\{-\alpha, 0\}$. The result still holds when $q^*=\infty$ and the inequalities in \eqref{eq2.06} are strict.
\end{lemma}

\begin{proof}
Let $w = x_d^\alpha u$, so that
\[
D_i w = x_d^\alpha D_i u + \delta_{id} \alpha x_d^{\alpha -1} u,
\]
where $\delta_{id} =1$ when $i =d$ and $\delta_{id} =0$ otherwise. By Lemma \ref{lem3.1} with $\alpha(q-1)$ in place of $\alpha$,  we have
\begin{equation*}
\|x_d^{\alpha-1} u\|_{L_q(B_2^+, d\mu_1)} \leq N \|x_d^\alpha D_d u\|_{L_q(B_2^+, d\mu_1)},
\end{equation*}
where $N = N(q, \alpha, d)$. Therefore
\begin{equation}
                                    \label{eq6.42}
\|Dw\|_{L_q(Q_2^+, d\mu_1)} \leq N \|x_d^\alpha D u\|_{L_q(Q_2^+, d\mu_1)}.
\end{equation}
Then, by applying the weighted Sobolev embedding \cite[Lemma 3.1]{DP20} to $w$ and using \eqref{eq6.42}, we obtain
\[
\begin{split}
&\|w\|_{L_{q^*} (Q_2^+, d\mu_1)}\\
& \leq N \Big[ \|w\|_{L_q(Q_2^+, d\mu_1)} + \|D w\|_{L_q(Q_2^+, d\mu_1)} + \|w_t\|_{\bH_q^{-1}(Q_1^+, d\mu_1)} \Big] \\
& \leq N \Big[ \|x_d^\alpha u\|_{L_q(Q_2^+, d\mu_1)} + \|x_d^\alpha D u\|_{L_q(Q_2^+, d\mu_1)} + \|x_d^\alpha u_t\|_{\bH_q^{-1}(Q_2^+, d\mu_1)} \Big]\\
& = N \|u\|_{\sH^1_q(Q_2^+, x_d^{\alpha q} d\mu_1)}.
\end{split}
\]
This implies \eqref{eq4.15bb} as desired.
\end{proof}

In the time-independent case, we also have the following embedding result in which the condition of $q$ and $q^*$ is optimal. 

\begin{remark}
        \label{imbd-remark}
Let $\alpha\in (-\infty, 1)$ and $q, q^*\in (1,\infty)$ satisfy
\begin{equation} \label{0811.eq}
\frac{1}{q} \le  \frac{1}{d+ \alpha_-}+ \frac{1}{q^*}.
\end{equation}
Then for any $u\in \sW^1_q(B_2^+, x_{d}^{\alpha q}d\mu_1)$, we have
\begin{equation*}
                   \|x_{d}^{\alpha} u\|_{L_{q^*}(B_2^+,d\mu_1)}\le N\|u\|_{\sW^1_{q}(B_2^+, x_{d}^{\alpha q}d\mu_1)},
\end{equation*}
where $N=N(d,\alpha, q, q^*)>0$. The result still holds when $q^*=\infty$ and the inequality in \eqref{0811.eq} is strict.
\end{remark}
\noindent 
The proof of this result is similar to that of Lemma \ref{lem2.2}. However, instead of applying \cite[Lemma 3.1]{DP20} as in the proof of  Lemma \ref{lem2.2}, we apply \cite[Remark 3.2 (ii)]{DP20}. 

Consider the parabolic equation
\begin{equation} \label{eq.23}
 x_d^\alpha \big(a_0(t,x) u_t + \lambda c_0(t,x) u \big) - D_i\big(x_d^\alpha (a_{ij}(t,x) D_j u - F_i)\big)   = \sqrt{\lambda} x_d^\alpha f
\end{equation}
in  $\Omega_T$ with the boundary condition
\begin{equation} \label{eq.23-bdr}
u(t,x', 0)= 0 \quad \text{for} \ (t, x') \in (-\infty, T) \times \bR^{d-1},
\end{equation}
where  $a_0, c_0: \Omega_T \rightarrow \mathbb{R}$ are given measurable functions satisfying
\begin{equation} \label{a-c-eq}
\kappa \leq a_0(t,x), \ c_0(t,x) \leq \kappa^{-1}, \quad (t,x) \in \Omega_T.
\end{equation}
Observe that \eqref{eq.23} is slightly different from \eqref{eq3.23} as there are non-constant coefficients ${a}_0$ and ${c}_0$. We introduce such coefficients because they will be useful to our future project on equations in non-divergence form.

\begin{lemma} \label{L-2-lemma} Let $\alpha \in (-\infty, 1)$,  $\lambda > 0$, and let $(a_{ij})$, $a_0$, and $c_0$ be measurable functions defined on $\Omega_T$ such that \eqref{ellipticity} and \eqref{a-c-eq} are satisfied. Then for each $F \in L_2(\Omega_T,d\mu)^d$ and $f \in  L_2(\Omega_T,d\mu)$, there exists a unique weak solution $u\in \sH^1_2(\Omega_T,d\mu)$ to \eqref{eq.23}-\eqref{eq.23-bdr}. Moreover,
\begin{equation} \label{L2-lemma-est}
\|D u\|_{L_2(\Omega_T,d\mu)}+\sqrt{\lambda} \|u\|_{L_2(\Omega_T,d\mu)}
\leq N \|F\|_{L_2(\Omega_T, d\mu)} +N\|f\|_{L_2(\Omega_T,d\mu)},
\end{equation}
where $N = N(\kappa)$.
\end{lemma}
\begin{proof}
The proof is standard and we give it here for completeness. We first prove the a priori estimate \eqref{L2-lemma-est}. Let $u\in \sH^1_2(\Omega_T,d\mu)$ be a weak solution of \eqref{eq.23}. By multiplying \eqref{eq.23} with $u$ (here as usual we need to apply the Steklov average) and using integration by parts and \eqref{ellipticity}, we obtain
\[
\begin{split}
& \sup_{t\in (-\infty, T)}\int_{\bR^d_+} |u(t,x)|^2 \,\mu(dx) + \int_{\Omega_T} |Du|^2 \,\mu(dz) + \lambda \int_{\Omega_T} |u(z)|^2 \,\mu(dz) \\
& \leq N\int_{\Omega_T} |F(z)| |Du(z)| \,\mu(dz) + N\lambda^{1/2} \int_{\Omega_T} |f(z)| |u(z)| \,\mu(dz).
\end{split}
\]
Then by Young's inequality, we obtain \eqref{L2-lemma-est}.

From \eqref{L2-lemma-est}, we see that the uniqueness follows. Now, to prove the existence of solution,  for each $k\in \mathbb{N}$, let
\begin{equation}
                \label{eq3.39}
\widehat{Q}_k = (-k^2 , \min\{k^2, T\}) \times B_k^+.
\end{equation}
We consider the equation
\begin{equation} \label{Oge-k.eqn}
x_d^\alpha (a_0 u_t+\lambda c_0 u) - D_i\big(x_d^\alpha (a_{ij} D_j u - F_i)\big)   =  \lambda^{1/2} x_d^\alpha f \quad \text{in} \  \widehat{Q}_k
\end{equation}
with the boundary conditions
\begin{equation} \label{bdr-Oge-k}
u = 0 \quad \text{on} \ \partial_p  \widehat{Q}_k,
\end{equation}
where $\partial_p  \widehat{Q}_k $ is the parabolic boundary of $\widehat{Q}_k$. By Galerkin's method, for each $k$, there exists a unique weak solution $u_k \in \sH_2^1(\widehat{Q}_k, d\mu)$ to  \eqref{Oge-k.eqn}-\eqref{bdr-Oge-k}. By taking $u_k =0$ on $\Omega_T \setminus  \widehat{Q}_k$, we also have
\[
\begin{split}
 &\sup_{t\in ((-\infty,T)} \|u_{k}(t,\cdot)\|_{L_2(\bR^{d}_+, d\mu)} + \|D u_k\|_{L_2(\Omega_T,d\mu)}
 +\lambda^{1/2}\|u_k\|_{L_2(\Omega_T,d\mu)}\\
& \leq N \|F\|_{L_2(\Omega_T, d\mu)} + N\|f\|_{L_2(\Omega_T,d\mu)}.
\end{split}
\]
By the weak compactness, there is a subsequence which is still denoted by $\{u_k\}$ and $u\in \sH^1_2(\Omega_T,d\mu)$ such that
$$
u_k\rightharpoonup u,\quad Du_k\rightharpoonup Du
$$
weakly in $L_2(\Omega_T,d\mu)$ as $k\rightarrow \infty$. By taking the limit in the weak formulation of solutions, it is easily seen that $u$ is a weak solution of  \eqref{eq3.23}. The lemma is proved.
\end{proof}

\section{Equations with simple coefficients}
                        \label{sec4}
In this section, we study the boundary value problem \eqref{eq3.23}-\eqref{m-bdr-cond} in which the coefficients only depend on the $x_d$-variable. We prove local pointwise estimates for gradients of solutions to homogeneous equations and the unique solvability of inhomogeneous equations in $\sH_p^1(\Omega_T, x_d^{\alpha p}d\mu_1)$.

Consider the parabolic equation
\begin{equation}
                    \label{eq11.52a}
x_d^\alpha (\overline{a}_0(x_d)u_t+\lambda \overline{c}_0(x_d) u) - D_i(x_d^\alpha (\overline{a}_{ij}(x_d)D_j u - F_i))=\sqrt{\lambda} x_d^\alpha f\quad \text{in}\ \Omega_T
\end{equation}
with the homogeneous Dirichlet boundary condition
\begin{equation}
                    \label{eq12.01a}
u=0 \quad \text{in}\ \partial\bR^{d+1}_+,
\end{equation}
where $\lambda\ge 0$, $\alpha\in (-\infty,1)$ are constants, and $\overline{a}_{ij}: \bR_+ \rightarrow \bR$ are measurable functions and satisfy the ellipticity condition: for some $\kappa\in (0,1)$,
\begin{equation} \label{cs-ellip}
\kappa |\xi|^2\le \overline{a}_{ij}(x_d)\xi_i\xi_j,\quad   \quad |\bar a_{ij}(x_d)|\le \kappa^{-1}
\end{equation}
for all $\xi = (\xi_1, \xi_2,\ldots, \xi_d) \in \bR^d, x_d \in \bR_+$, and $\overline{a}_0, \overline{c}_0 : \bR_+ \rightarrow \bR$ satisfy
\begin{equation} \label{a-b.zero}
\kappa \leq \overline{a}_0(x_d), \  \overline{c}_0(x_d) \leq \kappa^{-1}, \quad \forall \ x_d \in \bR_+.
\end{equation}
Here we introduce $\overline{a}_0$ and $\overline{c}_0$ again bearing in the mind the applications to future work about non-divergence form equations.

\subsection{Pointwise gradient estimates for homogeneous equations}
Let $\lambda \geq 0$, $r>0$, and $z_0 =(t_0, x_0) \in \overline{\bR^{d+1}_+}$.  In this subsection, we study \eqref{eq11.52a} in $Q_r^+(z_0)$ when $F= 0$, $f =0$, i.e., the homogeneous parabolic equation
\begin{equation}
                    \label{eq11.52}
x_d^\alpha ( \overline{a}_0(x_d) u_t + \lambda \overline{c}_0(x_d) u)
-D_i(x_d^\alpha \overline{a}_{ij}(x_d)D_j u)=0
\end{equation}
in $Q_r^+(z_0)$ with the homogeneous Dirichlet boundary condition: if  $B_r(x_0) \cap \partial \bR^d_+ \not= \emptyset$
\begin{equation}  \label{eq12.01}
        u(t,x', 0)=0, \quad (t,x', 0) \in  Q_{r}(z_0) \cap  (\bR \times \partial \overline{\bR^{d}_+}).
\end{equation}

The main goal in this subsection is to derive pointwise gradient estimates for weak solutions of \eqref{eq11.52}-\eqref{eq12.01}. See Propositions \ref{lem1} and \ref{interior-Linf} below.  Recall that $u \in \sH^{1}_2(Q_r^+(z_0), d\mu)$ is a weak solution of \eqref{eq11.52}-\eqref{eq12.01} if
\[
\int_{Q_r^+(z_0)} (-\overline{a}_0 u \varphi_t + \lambda \overline{c}_0 u \varphi)\, \mu(dz) +\int_{Q_{r}^+(z_0)} \overline{a}_{ij}(x_d) D_j u  D_i \varphi\, \mu(dz) = 0
\]
for all $\varphi \in C_0^\infty(Q_r^+(z_0))$.  Our first result is about the interior estimates of solutions to \eqref{eq11.52}.

\begin{proposition} \label{interior-Linf} Let $z_0 = (t_0, x_0) \in \Omega_T$ and suppose that $B_r(x_0) \subset \bR^d_+$. If $u \in \cH^1_{2}(Q_r^+(z_0), d\mu)$ is a weak solution to \eqref{eq11.52}, then we have
\begin{equation}
                    \label{eq5.23i}
|x_d^{\alpha}u(t,x)| \le N\Big(\fint_{Q^+_{2r/3}(z_0)}|\tilde{x}_d^{\alpha}u(\tilde{z})|^2 \, d\mu_1(\tilde{z})\Big)^{1/2}
\end{equation}
and
\begin{equation}
                    \label{eq5.24i}
 |{x_d^\alpha} Du(t,x)| \le N\Big(\fint_{Q^+_{2r/3}(z_0)}\big(|\tilde{x}_d^\alpha Du(\tilde{z})|^2  + \lambda |\tilde{x}_d^\alpha u(\tilde{z})|^2 \big)\, d{\mu_1}(\tilde{z})\Big)^{1/2}
\end{equation}
for any $(t,x) \in Q_{r/2}^+(z_0)$.
\end{proposition}

\begin{proof}
We write $x_0 = (x_0', x_{0d}) \in \bR^{d-1} \times \bR_+$ and by scaling, without loss of generality we may assume $r =1$. As $x_{0d} \geq 1$, the coefficients $x_d^\alpha \overline{a}_{ij}(x_d)$ is uniformly elliptic in $Q_{2/3}(z_0)$. Then, from the standard energy estimates and the Sobolev embedding (see, for instance, \cite[Lemma 3.5]{DK11}), we obtain
\[
\|u\|_{L_\infty(Q_{1/2}(z_0))}\le N\Big(\fint_{Q_{2/3}(z_0)}|u|^2 \, dx \Big)^{1/2}
\]
and
\[
\| Du\|_{L_\infty(Q_{1/2}(z_0))}\le N\Big(\fint_{Q_{2/3}(z_0)}\big(|Du|^2  + \lambda | u|^2 \big)\, d x \Big)^{1/2}.
\]
From this and $x_d \sim x_{d0}$  in $Q_{2/3}(z_0)$, we obtain \eqref{eq5.23i} and \eqref{eq5.24i}.
\end{proof}
The next result is about boundary pointwise gradient estimates of solutions.

\begin{proposition}
        \label{lem1}
         Let  $u\in \sH^1_{2}(Q_1^+, d\mu)$ be a weak solution to \eqref{eq11.52}-\eqref{eq12.01} in $Q_1^+$. Then we have
\begin{equation}
                    \label{eq5.23}
x_d^{\alpha}|u(t,x)|\le N  x_d \Big(\fint_{Q^+_{1}}|\tilde{x}_d^{\alpha} u (\tilde{z})|^2 \, \mu_1(d\tilde{z})\Big)^{1/2}
\end{equation}
and
\begin{equation}
                    \label{eq5.24}
x_d^\alpha |Du(t,x)|\le N \Big(\fint_{Q^+_{1}}\big(|\tilde{x}_d^\alpha Du(\tilde{z})|^2  + \lambda |\tilde{x}_d^\alpha u(\tilde{z})|^2 \big)\, \mu_1(d\tilde{z})\Big)^{1/2}
\end{equation}
for any $(t,x) \in Q_{1/2}^+$.
\end{proposition}

\begin{proof}
We  adapt the approach in \cite{DP19} which works also for parabolic systems. For $0 < r <R \leq 1$, testing \eqref{eq11.52} by $u\varphi^2$, where $\varphi \in C_0^\infty(Q_R)$ satisfying $\varphi \equiv 1$ in $Q_r$, we obtain the following Caccioppoli inequality
\begin{equation} \label{eq11.58}
\int_{Q_r^+} \Big( |Du|^2  + \lambda |u|^2 \Big)\,{\mu}(dz)  \leq N (d, \kappa,r , R) \int_{Q_R^+} |u|^2 \,{\mu}(dz).
\end{equation}
Similar to the proof of \cite[Lemma 4.2]{DP19}, by testing the equation with $u_t\varphi^2$, we obtain
$$
\int_{Q_r^+} |u_t|^2 \,{\mu}(dz)  \leq N (d, \kappa,r , R) \int_{Q_R^+}  \Big( |Du|^2  + \lambda |u|^2 \Big)\,{\mu}(dz).
$$
Moreover, by using the difference quotient method in the $t$ and $x'$ variables, we also have
\begin{align} \nonumber
& \int_{Q_r ^+}|\partial_t^{j+1}u|^2 \,{\mu}(dz) + \int_{Q_r ^+} |DD_{x'}^k \partial_t^ju|^2 \,{\mu}(dz) \\ \label{inter-eq3.50}
& \le N(d,  \kappa, k, j, r, R) \int_{Q_{R}^+} \Big( {|Du|^2 + \lambda |u|^2} \Big)\,{\mu}(dz)
\end{align}
for any $k, j  \in \mathbb{N} \cup \{0\}$. We now prove that
\begin{align} \label{0727.eqn}
&|u_t(z)|+ |D_{x'}u(z)|\notag\\
&\le N x_d^{(1-\alpha)/2} \Big[\|Du\|_{L_2(Q_1^+, d\mu)} + \sqrt{\lambda} \|u\|_{L_2(Q_1^+, d\mu)} \Big]
\end{align}
for any $z\in Q^+_{1/2}$.
By applying the Sobolev embedding theorem in $z'=(t,x')$, for any $z=(z',x_d)\in Q_{1/2}^+$,
we have
\begin{align*}
 |D_du(z',x_d)|\le  N(d) \|D_du(\cdot, x_d)\|_{W^{k/2,k}_2(Q_{1/2}')}
\end{align*}
with an even integer $k> (d+1)/2$. Then, using \eqref{eq11.58} and \eqref{inter-eq3.50},  we have
\begin{align}  \label{eq4.16bb}
\int_{0}^{1/2} x_d^{\alpha} |D_d u(z',x_d)|^2 \, dx_d
&\le N \int_{0}^{1/2} x_d^{\alpha}  \|D_d u(\cdot, x_d)\|^2_{W^{k/2,k}_2(Q_{1/2}')}\,dx_d\nonumber\\
&\le N \|u\|_{L_2(Q_1^+, d\mu)}^2, \quad z'\in Q_{1/2}'.
\end{align}
From this, and  by H\"older's inequality and \eqref{eq4.16bb}, we infer that
\begin{align*}
&\int_{0}^{x_d} |D_d u(z', \tilde x_d)|\, d\tilde x_d \\
&\le \left(\int_{0}^{1/2}x_d^{\alpha}|D_d u(z',\tilde x_d)|^2 \, d\tilde x_d\right)^{1/2} \left(
\int_{0}^{x_d} \tilde x_d^{-\alpha}\,d\tilde x_d \right)^{1/2}\\
&\le N x_d^{(1-\alpha)/2} \|u\|_{L_2(Q_1^+, d\mu)},
\end{align*}
where we also used $\alpha<1$ in the last inequality. By the fundamental theorem of calculus and the boundary condition $u(x', 0) =0$, we obtain
\begin{equation} \label{0727.eqn1}
|u(z', x_d)| \leq \int_0^{x_d} |D_du(z', s)|\, ds \leq N x_d^{(1-\alpha)/2}\|u\|_{L_2(Q_1^+, d\mu)}, \quad \forall\,  z \in Q_{1/2}^+.
\end{equation}
Now, applying the difference quotient method if needed, we see that $D_{x'}u$ and $u_t$ solve the same equation as $u$.  Then,  we apply \eqref{0727.eqn1} to $D_{x'}u$ and $u_t$ and then use \eqref{inter-eq3.50} to obtain \eqref{0727.eqn}.

Next, let
\begin{equation}
                \label{eq11.31}
\cU= \overline{a}_{dj}(x_d)D_j u.
\end{equation}
By using the Sobolev inequality in the $z'$-variable, we see that
\begin{align*}
& |\cU(z', x_d)| +  |u_t(z', x_d)|  + |DD_{x'}u(z', x_d)| \\
& \leq N\Big[\|\cU(\cdot, x_d)\|_{W^{k/2, k}_2(Q_{1/2}')} + \|u_t(\cdot, x_d)\|_{W^{k/2, k}_2(Q_{1/2}')}  \\
& \qquad \quad +  \|D D_{x'} u(\cdot, x_d)\|_{W^{k/2, k}_2(Q_{1/2}')} \Big]
\end{align*}
for even $k > (d+1)/2$ and $z' \in Q_{1/2}'$. Then,
\begin{align} \nonumber
&  \int_0^{1/2}x_d^\alpha \Big( |\cU(z', x_d)|^2 +  |u_t(z', x_d)|^2  + |DD_{x'}u(z', x_d)|^2\Big)\, dx_d\\ \nonumber
&\leq N\int_0^{1/2}x_d^\alpha \Big(\|\cU(\cdot, x_d)\|_{W^{k/2, k}_2(Q_{1/2}')}^2 + \|u_t(\cdot, x_d)\|_{W^{k/2, k}_2(Q_{1/2}')}^2 \\\ \nonumber
& \qquad \qquad  +  \|D D_{x'} u(\cdot, x_d)\|_{W^{k/2, k}_2(Q_{1/2}')}^2 \Big) \, dx_d \\ \label{0728.eq1}
& \leq N\big( \|Du\|^2_{L_2(Q_1^+, d\mu)} + \lambda \|u\|^2_{L_2(Q_1^+, d\mu)}\big), \quad  z' \in Q_{1/2}',
\end{align}
where we used \eqref{inter-eq3.50} and \eqref{eq11.58} in the last estimate. Note that from \eqref{eq11.52},
\begin{equation}
                    \label{eq12.46}
|D_d (x_d^\alpha \cU)| \leq  Nx_d^{\alpha}[|u_t| + \lambda |u| + |DD_{x'} u| ].
\end{equation}
Then, it follows from the last estimate, H\"older's inequality, \eqref{0728.eq1}, and \eqref{eq11.58} that
\begin{align} \notag
&|x_d^\alpha\cU(z',x_d)-2^{-\alpha}\cU(z',1/2)|  \label{eq11.01} \\
&\le N\int_{x_d}^{1/2} \tilde x_d^\alpha\big(|u_t(z',\tilde x_d)|
+\lambda |u(z',\tilde x_d)|+|DD_{x'}u(z',\tilde x_d)|\big)\,d\tilde x_d  \\
&\le N\big(1+x_d^{\frac{1+\alpha}2-\varepsilon}\big)\Big(\int_{x_d}^{1/2} \tilde x_d^\alpha\big(|u_t(z',\tilde x_d)|
+\lambda |u(z',\tilde x_d)|+|DD_{x'}u(z',\tilde x_d)|\big)^2\,d\tilde x_d\Big)^{1/2}  \notag  \\
&\le N\big(1+x_d^{\frac{1+\alpha}2-\varepsilon}\big)\big( \|Du\|_{L_2(Q_1^+, d\mu)} + \sqrt\lambda \|u\|_{L_2(Q_1^+, d\mu)}\big)  \notag
\end{align}
for any small $\varepsilon>0$, which is included in order to avoid the $\log$ correction when $\alpha =-1$.  This together with the interior gradient estimate \eqref{eq5.24i} of Proposition \ref{interior-Linf} gives
$$
|x_d^\alpha\cU(z',x_d)|
\le N\big(1+x_d^{\frac{1+\alpha}2-\varepsilon}\big)\big( \|Du\|_{L_2(Q_1^+, d\mu)} + \sqrt\lambda \|u\|_{L_2(Q_1^+, d\mu)}\big).
$$
Therefore, by using the ellipticity condition \eqref{cs-ellip}, the definition of $\cU$, and \eqref{0727.eqn}, we obtain
\begin{equation}
                        \label{eq10.56}
|D_d u(z',x_d)|\le N\big(x_d^{-\alpha}+x_d^{\frac{1-\alpha}2-\varepsilon}\big)\big( \|Du\|_{L_2(Q_1^+, d\mu)} + \sqrt\lambda \|u\|_{L_2(Q_1^+, d\mu)}\big).
\end{equation}
By using the zero boundary condition, we have
\begin{equation}
                                    \label{eq11.00}
|u(z',x_d)|\le N\big(x_d^{1-\alpha}+x_d^{1+\frac{1-\alpha}2-\varepsilon}\big)\big( \|Du\|_{L_2(Q_1^+, d\mu)} + \sqrt\lambda \|u\|_{L_2(Q_1^+, d\mu)}\big).
\end{equation}
Since $D_{x'}u$ and $u_t$ satisfy the same equation as $u$, by using the above estimate and \eqref{inter-eq3.50}, we get
\begin{align}
                        \label{eq11.02}
&|u_t(z',x_d)|+|D_{x'}u(z',x_d)|\notag\\
&\le N\big(x_d^{1-\alpha}+x_d^{1+\frac{1-\alpha}2-\varepsilon}\big)\big( \|Du\|_{L_2(Q_1^+, d\mu)} + \sqrt\lambda \|u\|_{L_2(Q_1^+, d\mu)}\big),
\end{align}
which together with \eqref{eq10.56} gives
\begin{equation}
                        \label{eq10.56b}
|D u(z',x_d)|\le N\big(x_d^{-\alpha}+x_d^{\frac{1-\alpha}2-\varepsilon}\big)\big( \|Du\|_{L_2(Q_1^+, d\mu)} + \sqrt\lambda \|u\|_{L_2(Q_1^+, d\mu)}\big).
\end{equation}
Again, because $D_{x'}u$ satisfies the same equation as $u$, from \eqref{eq10.56b}, \eqref{inter-eq3.50}, and \eqref{eq11.58}, we get
\begin{equation}
                        \label{eq10.56c}
|DD_{x'} u(z',x_d)|\le N\big(x_d^{-\alpha}+x_d^{\frac{1-\alpha}2-\varepsilon}\big)\big( \|Du\|_{L_2(Q_1^+, d\mu)} + \sqrt\lambda \|u\|_{L_2(Q_1^+, d\mu)}\big).
\end{equation}
Feeding \eqref{eq11.00}, \eqref{eq11.02}, and \eqref{eq10.56c} back to \eqref{eq11.01} yields
\begin{align*}
&|x_d^\alpha\cU(z',x_d)-2^{-\alpha}\cU(z',1/2)|\\
&\le N\big(1+x_d^{1+\frac{1+\alpha}2-\varepsilon}\big)\big( \|Du\|_{L_2(Q_1^+, d\mu)} + \sqrt\lambda \|u\|_{L_2(Q_1^+, d\mu)}\big).
\end{align*}
By using the interior gradient estimate \eqref{eq5.24i} and \eqref{eq11.02}, we get
\begin{equation*}
              %          \label{eq11.10}
|D_d u(z',x_d)|\le N\big(x_d^{-\alpha}+x_d^{1+\frac{1-\alpha}2-\varepsilon}\big)\big( \|Du\|_{L_2(Q_1^+, d\mu)} + \sqrt\lambda \|u\|_{L_2(Q_1^+, d\mu)}\big),
\end{equation*}
which improves \eqref{eq10.56}. Similar to \eqref{eq11.00} and \eqref{eq11.02}, we also have
\begin{align*}
&|u(z',x_d)|+|u_t(z',x_d)|+|D_{x'}u(z',x_d)|\\
&\le N\big(x_d^{1-\alpha}+x_d^{2+\frac{1-\alpha}2-\varepsilon}\big)\big( \|Du\|_{L_2(Q_1^+, d\mu)} + \sqrt\lambda \|u\|_{L_2(Q_1^+, d\mu)}\big).
\end{align*}
By iteration, in finitely many steps we reach
$$
|D_d u(z',x_d)|\le Nx_d^{-\alpha}\big( \|Du\|_{L_2(Q_1^+, d\mu)} + \sqrt\lambda \|u\|_{L_2(Q_1^+, d\mu)}\big)
$$
and
\begin{align*}
&|u(z',x_d)|+|u_t(z',x_d)|+|D_{x'}u(z',x_d)|\\
&\le Nx_d^{1-\alpha}\big( \|Du\|_{L_2(Q_1^+, d\mu)} + \sqrt\lambda \|u\|_{L_2(Q_1^+, d\mu)}\big),
\end{align*}
which imply \eqref{eq5.23} and \eqref{eq5.24}. The proposition is proved.
\end{proof}

\subsection{Solvability of solutions}
In this subsection, we prove Theorem \ref{thm3} below about the existence and uniqueness of solutions to \eqref{eq11.52a}-\eqref{eq12.01a}. This theorem can be considered as a simplified version of Theorem \ref{thm2} and it will be used later in the proof of Theorem \ref{thm2}.

\begin{theorem}
            \label{thm3}
Let $\alpha \in (-\infty, 1)$ and  $\lambda >0$.  Suppose that \eqref{cs-ellip} and  \eqref{a-b.zero} are satisfied. Then the following assertions hold.
\\
\noindent\textup{\bf{(i)}} Suppose that $F : \Omega_T \rightarrow \bR^d$ and $f: \Omega_T \rightarrow \bR$ such that $|F|+ |f| \in L_p(\Omega_T, x_d^{\alpha p} d\mu_1)$ for $p\in (2, \infty)$. Then, for every weak solution $u \in \sH_{q}^1(\Omega_T,  x_d^{\alpha q}d\mu_1)$  of \eqref{eq11.52a}-\eqref{eq12.01a} for some $q\in [2,p]$, we have  $u\in \sH_{p}^1(\Omega_T,  x_d^{\alpha p}d\mu_1)$ and
\begin{equation} \label{thm3-est}
\begin{split}
& \|Du\|_{L_p(\Omega_T, x_d^{\alpha p} d\mu_1)}
+ \sqrt{\lambda}\|u\|_{L_p(\Omega_T, x_d^{\alpha p} d\mu_1)}\\
& \leq N \Big[ \|F\|_{L_p(\Omega_T, x_d^{\alpha p} d\mu_1)}
+ \|f\|_{L_p(\Omega_T, x_d^{\alpha p}d\mu_1)}\Big],
\end{split}
\end{equation}
where $N = N(d, \alpha, \kappa, p)$.\\
\noindent\textup{\bf{(ii)}} For each $F : \Omega_T \rightarrow \bR^d$ and $f: \Omega_T \rightarrow \bR$ such that $|F|+ |f| \in L_p(\Omega_T, x_d^{\alpha p} d\mu_1)$ with $p\in (1, \infty)$, there exists unique weak solution $u \in \sH_p^1(\Omega_T, x_d^{\alpha p} d\mu_1)$ of \eqref{eq11.52a}-\eqref{eq12.01a}. Moreover, \eqref{thm3-est} holds.
\end{theorem}
%=====
The remaining part of the section is to prove this theorem. We begin with the following result on solution decomposition which is an important ingredient in the proof.
\begin{proposition} \label{Simple-approx} Let $z_0\in \overline{\Omega_T}$ and $r >0$. Suppose that $F \in L_{2}(Q_{10r}^+(z_0), d\mu)^d$, $f \in L_{2}(Q_{10r}^+(z_0), d\mu)$, and $u \in \cH^{1}_2(Q_{10r}^+(z_0),d\mu)$ is a weak solution of \eqref{eq11.52a}-\eqref{eq12.01a} in $Q_{10r}^+(z_0)$. Then we can write
\[
u(t, x) = v(t, x) + w(t, x) \quad  \text{in}\ Q_{10r}^+(z_0),
\]
where $v$ and $w$ are functions in $\cH_2^1(Q_{10r}^+(z_0), d\mu)$ and satisfy
\begin{equation}  \label{0506-tU-est}
\fint_{Q_{2r}^+(z_0)}|V|^2  \,{\mu}_1(dz)  \leq N \fint_{Q_{10r}^+(z_0)}\Big( |x_d^\alpha F|^2 +  |x_d^\alpha f|^2) \,{\mu}_1(dz)
\end{equation}
and
\begin{align} \nonumber
\|W\|_{L_\infty(Q_{r}^+(z_0))}^{2}
& \leq N  \fint_{Q_{10r}^+(z_0)} |U|^2 \,{\mu}_1(dz) \\ \label{0506-W.est}
& \qquad + N \fint_{Q_{10r}^+(z_0)} \Big(|x_d^\alpha F|^2 + |x_d^\alpha f|^2\Big) \,{\mu}_1(dz),
\end{align}
where $N = N(d,\kappa, \alpha)$ and
\[
 V=x_d^\alpha(|Dv|+\lambda^{1/2}|v|),\quad W=x_d^\alpha(|Dw|+\lambda^{1/2}|w|),\quad U=x_d^\alpha(|D u|+\lambda^{1/2}|u|).
\]
\end{proposition}
\begin{proof} We write $z_0 = (t_0, x_0)$ with $x_0 = (x_0', x_{0d}) \in \bR^{d-1} \times \bR_+$. We split the proof into the interior case and the boundary case.

\noindent{\bf Case I}. Consider $x_{0d} > 2r$. Let $v \in \sH^1_2(\Omega_T, d\mu)$ be a weak solution of the equation
\begin{equation*}
\begin{split}
& x_d^\alpha  ( {\overline{a}_0(x_d)}v_t  +\lambda {\overline{c}_0(x_d)} v)  - D_i
\big(x_d^\alpha(\overline{a}_{ij}(x_d) D_j v - F_{i}(z)\chi_{Q_{2r}^+(z_0)}(z))\big)  \\
& =  \lambda^{1/2} x_d^\alpha f (z) \chi_{Q_{2r}^+(z_0)}(z)  \quad \text{in} \  \Omega_T
\end{split}
\end{equation*}
with the boundary condition $ v=  0$ on $(-\infty, T) \times \{x_d =0\}$. Then \eqref{0506-tU-est} follows from Lemma \ref{L-2-lemma} and the doubling property of $\mu_1$. Now let $w = u - v$ so that $w \in \cH_2^1(Q_{2r}^+(z_0) , d\mu)$ is a weak solution of
\[
x_d^\alpha  (\overline{a}_0(x_d) w_t +\lambda \overline{c}_0(x_d) w)   - D_i \big(x_d^\alpha \overline{a}_{ij}(x_d) D_j  w\big)   = 0 \quad \text{in} \  Q_{2r}^+(z_0).
\]
By Proposition \ref{interior-Linf} and the triangle inequality, we obtain
\[
\|W\|_{L_\infty(Q_{r}^+(z_0))}^{2}
\leq N  \fint_{Q_{2r}^+(z_0)} |U|^2 \,{\mu}_1(dz) + N \fint_{Q_{2r}^+(z_0)} \Big(|x_d^\alpha F|^2 + |x_d^\alpha f|^2\Big) \,{\mu}_1(dz).
\]
From this, we get \eqref{0506-W.est} by using the doubling property of $\mu_1$.

\noindent{\bf Case II}. Consider $x_{0d} \le 2r$. Let $\hat{z}_0 = (t_0, x_0', 0)$ and  $v \in \sH^1_2(\Omega_T, d\mu)$ be a weak solution of the equation
\begin{equation*}
\begin{split}
& x_d^\alpha  ( {\overline{a}_0(x_d)}v_t  +\lambda {\overline{c}_0(x_d)} v)  - D_i
\big(x_d^\alpha(\overline{a}_{ij}(x_d) D_j v - F_{i}(z)\chi_{Q_{8r}^+(\hat{z}_0)}(z))\big)  \\
& =  \lambda^{1/2} x_d^\alpha f (z) \chi_{Q_{8r}^+(\hat{z}_0)}(z)  \quad \text{in} \  \Omega_T
\end{split}
\end{equation*}
with the boundary condition $v =0$ on $(-\infty, T) \times\{x_d =0\}$. Then, it follows from Lemma \ref{L-2-lemma} that
\begin{equation} \label{V-L-2-0716}
\fint_{Q_{8r}^+(\hat z_0)}|V|^2  \,{\mu}_1(dz)  \leq N \fint_{Q_{8r}^+(\hat z_0)}\Big( |x_d^\alpha F|^2 +  |x_d^\alpha f|^2) \,{\mu}_1(dz).
\end{equation}
As $Q_{2r}^+(z_0) \subset Q_{8r}^+(\hat{z}_0) \subset Q_{10r}^+(z_0)$,  \eqref{0506-tU-est} follows from \eqref{V-L-2-0716} and the doubling property of $\mu_1$.

Now, let $w = u - v$ so that $w \in \sH_2^1(Q_{8r}^+(\hat{z}_0), d\mu)$ is a weak solution of
\[
x_d^\alpha  (\overline{a}_0(x_d) w_t +\lambda \overline{c}_0(x_d) w)   - D_i \big(x_d^\alpha \overline{a}_{ij}(x_d) D_j  w\big)   = 0 \quad \text{in} \  Q_{8r}^+(\hat{z}_0)
\]
with boundary condition $w = 0$ on $\{x_d =0\} \cap \overline{Q^+_{8r}(\hat{z}_0)}$. Then, applying Proposition \ref{lem1} with suitable scaling, the triangle inequality, and \eqref{V-L-2-0716}, we obtain
\[
\begin{split}
& \|W\|_{L_\infty(Q_{4r}^+(\hat{z}_0))} \leq N\left(\fint_{Q_{8r}^+(\hat{z}_0)} |W|^2 \mu_1(dz) \right)^{1/2} \\
& \leq N\left(\fint_{Q_{8r}^+(\hat{z}_0)} |U|^2 \mu_1(dz) \right)^{1/2} + N \left(\fint_{Q_{8r}^+(\hat{z}_0)} |V|^2 \mu_1(dz) \right)^{1/2} \\
& \leq N\left(\fint_{Q_{8r}^+(\hat{z}_0)} |U|^2 \mu_1(dz) \right)^{1/2} +  N\left(\fint_{Q_{8r}^+(\hat{z}_0)}( |x_d^\alpha  F|^2  + |x_d^\alpha f|^2) \mu_1(dz) \right)^{1/2}.
\end{split}
\]
Then, \eqref{0506-W.est} follows as $Q_{2r}^+(z_0) \subset Q_{4r}^+(\hat{z}_0) \subset Q_{8r}^+(\hat{z}_0) \subset Q_{10r}^+(z_0)$. The proof of the proposition is completed.
\end{proof}
\begin{proof}[Proof of Theorem \ref{thm3}]
We use an idea which is similar to that of \cite[Theorem 4.1]{DP20}. For $p \in (2, \infty)$, we use a real variable argument by applying Proposition \ref{Simple-approx}. For $p \in (1, 2)$, we use a duality argument. Nevertheless, some details need to be carried out to adapt the proof of \cite[Theorem 4.1]{DP20} to our case. For completeness, we present them in Appendix \ref{proof-thm4.6}.
\end{proof}
%=======
\section{Equation with measurable coefficients}
                                            \label{sec5}
In this section, we give the proofs of Theorem \ref{thm2}, Corollary \ref{main-thrm}, and Theorem \ref{thm1.4}. For Theorem \ref{thm2}, we apply the level set argument introduced in \cite{Cal}. The proof of Corollary \ref{main-thrm} follows from a localization technique and the duality argument adapting the ideas in \cite{KRW20}. For the proof of Theorem \ref{thm1.4}, we apply the perturbation technique using the method of mean oscillation estimates introduced in \cite{Krylov} and developed in \cite{MR3812104}.
%====
\subsection{Proof of Theorem \ref{thm2}} We begin with the following proposition that is similar to Proposition \ref{Simple-approx}.
%====
\begin{proposition} \label{G-approx-propos} Let $\delta_0 \in (0, 1)$, $\alpha \in (-\infty, 1)$,  $r \in (0, \infty)$,  $z_0\in \overline{\Omega_T}$, and $q \in (2,\infty)$. Suppose that $G=  x_d^{\alpha}(|F| + |f|) \in L_{2}(Q_{10r}^+(z_0), d\mu_1)$ and
$$
u \in \sH^{1}_q(Q_{10r}^+(z_0),x_d^{\alpha q}d\mu_1)
$$
is a weak solution of \eqref{eq3.23}-\eqref{m-bdr-cond} in $Q_{10r}^+(z_0)$. If \textup{Assumption \ref{assump1} ($\delta_0, R_0$)} is satisfied and $\textup{spt}(u) \subset   (s - (R_0r_0)^2, s + (R_0r_0)^2) \times \bR^{d}_+$ for some $r_0>0$ and $s \in \bR$, then we have
\[
u(t, x) = v(t, x) + w(t, x) \quad  \text{in}\  Q_{10r}^+(z_0),
\]
where $v$ and $w$ are functions in $\sH_2^1(Q_{10r}^+(z_0), d\mu)$ that satisfy
\begin{align} \nonumber
\fint_{Q_{2r}^+(z_0)} |V|^2 \,{\mu_1}(dz) &  \leq N \fint_{Q_{10r}^+(z_0)} |G|^{2} \,{\mu_1}(dz)   \\ \label{B-u-tilde-est-inter}
& \quad +  N(\delta_0^{1-2/q}+ r_0^{2-4/q}) \left(\fint_{Q_{10r}^+(z_0)} |x_d^{\alpha}Du|^q \,{\mu_1}(dz) \right)^{2/q}
\end{align}
and
\begin{align} \label{D-L-infty-w-inter}
\|W\|_{L_\infty(Q_{r}^+(z_0))}^{2} \leq N  \fint_{Q_{10r}^+(z_0)} |U|^{2}\,{\mu_1}(dz) +  N\fint_{Q_{10r}^+(z_0)}  |G|^{2} \,{\mu_1}(dz),
\end{align}
where
$$
V=x_d^{\alpha}(|Dv|+ \sqrt{\lambda}|v|),\quad W=x_d^{\alpha} (|Dw|+\sqrt{\lambda}|w|),\quad U=x_d^{\alpha}(|D u|+\sqrt{\lambda}|u|),
$$
and
$N = N(d,\alpha, \kappa,  q)$.
\end{proposition}
\begin{proof} Let $\tilde F = (\tilde F_1, \tilde F_2,\ldots, \tilde F_d)$, where
\[
\tilde F_{i}(t,x) = \big( \overline{a}_{ij}(x_d)-a_{ij}(t,x)\big)D_j u (t,x),
\]
where $\overline{a}_{ij}(x_d)= [a_{ij}]_{10r, {z}_0}(x_d)$ are defined in Assumption \ref{assump1}.

If $r \in (0, R_0/10)$,  by H\"{o}lder's inequality, the boundedness of the matrix $(a_{ij})$ in \eqref{ellipticity},  and Assumption \ref{assump1} ($\delta_0, R_0$), we have
\begin{align*}
 & \fint_{Q_{10r}^+(z_0)} |x_d^{\alpha}\tilde F(z)|^2 \,{\mu_1}(dz)  \\
& \leq \left(\fint_{Q_{10r}^+(z_0)} |a_{ij} -\overline{a}_{ij}(x_d)|^{\frac{2q}{q-2}} \,{\mu_1}(dz) \right)^{\frac{q-2}{q}} \left(\fint_{Q_{10r}^+(z_0)} |x_d^{\alpha}D u|^{q} \,{\mu_1}(dz) \right)^{\frac{2}{q}} \\ \nonumber
& \leq N \delta_0^{\frac{q-2}{q}} \left(\fint_{Q_{10r}^+(z_0)} |x_d^{\alpha} D u |^q \,{\mu_1}(dz) \right)^{2/q}.
\end{align*}
On the other hand, when $r \geq R_0/10$, as $\text{spt}(u) \subset  (s - (R_0r_0)^2, s + (R_0r_0)^2) \times \bR^{d}_+$ and by using the boundedness of the matrix $(a_{ij})$ in \eqref{ellipticity}, we have
\begin{align*} % \label{G-inter-est}
& \fint_{Q_{10r}^+(z_0)} |x_d^{\alpha}\tilde F(z)|^2 \,{\mu_1}(dz)  \\
& \leq N\left(\fint_{Q_{10r}^+(z_0)}
\chi_{(s - (R_0r_0)^2, s + (R_0r_0)^2)}(t)  \,{\mu_1}(dz) \right)^{\frac{q-2}{q}} \left(\fint_{Q_{10r}^+(z_0)} |x_d^{\alpha}D u|^{q} \,{\mu_1}(dz) \right)^{\frac{2}{q}} \\
& \leq N \Big(\frac{R_0 r_0}{r} \Big)^{\frac{2(q-2)}{q}} \left(\fint_{Q_{10r}^+(z_0)} |x_d^{\alpha} D u |^q \,{\mu_1}(dz) \right)^{2/q} \\
& \leq N r_0^{\frac{2(q-2)}{q}} \left(\fint_{Q_{10r}^+(z_0)} |x_d^\alpha D u |^q \,{\mu_1}(dz) \right)^{2/q}.
\end{align*}
Therefore, in both cases we have
\begin{align} \nonumber
&  \fint_{Q_{10r}^+(z_0)} |x_d^{\alpha}\tilde F(z)|^2\, {\mu_1}(dz)\\   \label{G-inter-est}
& \leq N \Big(r_0^{\frac{2(q-2)}{q}} + \delta_0^{\frac{q-2}{q}}\Big)
 \left(\fint_{Q_{10r}^+(z_0)} |x_d^{\alpha} D u |^q \,{\mu_1}(dz) \right)^{2/q}.
\end{align}
Since $u \in \sH_{q}^1(Q_{10r}^+(z_0), x_d^{\alpha q} d\mu_1)$ is a weak solution of
$$
x_d^\alpha  (\partial_t u +\lambda u) -  D_i \big(x_d^\alpha(\overline{a}_{ij}(x_d) D_{j} u-\tilde F_i-F_i) \big)   = \lambda^{1/2} x_d^\alpha f
$$
in $Q_{10r}^+(z_0)$ and \eqref{m-bdr-cond}, applying Proposition \ref{Simple-approx} with $\tilde F+F$ in place of $F$ and using \eqref{G-inter-est}, we obtain \eqref{B-u-tilde-est-inter} and \eqref{D-L-infty-w-inter}. The proposition is proved.
\end{proof}
%====
\begin{proof}[Proof of Theorem \ref{thm2}]
We only need to prove Theorem \ref{thm2} when $p \in (2, \infty)$ as the case $p \in (1,2)$ can be  proved by  using the duality argument as in the proof of Theorem \ref{thm3}. See Appendix \ref{proof-thm4.6}. We first prove the estimate \eqref{main-thm-est} for each weak solution $u \in \sH_p^1(\Omega_T, x_d^{\alpha p}d\mu_1)$ of \eqref{eq3.23}. We suppose that $\lambda>0$.
Assume for a moment that
$$
\textup{spt}(u) \subset  (s - (R_0r_0)^2, s + (R_0r_0)^2) \times \bR^{d}_+
$$
with some $s \in (-\infty, T)$ and $r_0 \in (0,1)$. We claim that  \eqref{main-thm-est} holds if $\delta_0$ and  $r_0$ are sufficiently small depending on $d$, $\alpha$, $\kappa$, and $p$. Let $q \in (2, p)$ be fixed.  By H\"{o}lder's inequality and using $\alpha<1$, we  have $u \in \sH_{q, \text{loc}}^1(\Omega_T, x_d^{\alpha q} d\mu_1)$. Applying Proposition \ref{G-approx-propos}, for each $r >0$ and $z_0 \in \overline{\Omega_T}$, we can write
\[
u(t, x) = v(t, x) + w(t, x) \quad  \text{in}\ Q_{10r}^+(z_0),
\]
where $v$ and $w$ satisfy \eqref{B-u-tilde-est-inter} and \eqref{D-L-infty-w-inter}.  Then it follows from the standard real variable argument (see, for example, \cite{DK11b} and \cite[Lemma A.20]{DK19}) that
\[
\begin{split}
& \|x_d^{\alpha}Du\|_{L_p(\Omega_T, d\mu_1)} + \sqrt{\lambda} \|x_d^{\alpha} u\|_{L_p(\Omega_T, d\mu_1)}  \\
& \leq N(\delta_0^{1-2/q} + r_0^{2-4/q}) \|x_d^{\alpha} Du\|_{L_{p}(\Omega_T, d\mu_1)} + N\|x_d^{\alpha} (|F|+|f|)\|_{L_p(\Omega_T, d\mu_1)},% +N\|x_d^{\alpha} f\|_{L_p(\Omega_T, d\mu_1)}
\end{split}
\]
where $N = N(d, \alpha, \kappa, p)$. From this, and by choosing $\delta_0$ and $r_0$ sufficiently small so that $N(\delta_0^{1-2/q} + r_0^{2-4/q}) < 1/2$, we obtain \eqref{main-thm-est}.

Now, we remove the assumption that $\textup{spt}(u) \subset  (s -( R_0r_0)^2, s+ (R_0r_0)^2) \times \bR^{d}_+$ by using a partition of unity argument. The proof is standard, but the details are slightly different so we give them here. Let
$$
\xi=\xi(t) \in C_0^\infty(-(R_0r_0)^2, (R_0r_0)^2)
$$
be a standard non-negative cut-off function satisfying
\begin{equation} \label{xi-0515}
\int_{\bR} \xi^p(s)\, ds =1, \quad  \int_{\bR}|\xi'(s)|^p\,ds \leq \frac{N}{(R_0r_0)^{2p}}.
\end{equation}
For any $s \in (-\infty,  \infty)$, let $u^{(s)}(z) = u(z) \xi(t-s)$ for $z = (t, x) \in \Omega_T$. Then $u^{(s)} \in \sH_p^1(\Omega_T, x_{d}^{\alpha p}d\mu_1)$ is a weak solution of
\[
x_d^\alpha( u^{(s)}_t + \lambda u^{(s)}) -D_i\big(x_d^\alpha(a_{ij} D_j u^{(s)} - F^{(s)}_{i})\big)  = \lambda^{1/2} x_d^\alpha f^{(s)}\]
in $\Omega_T$ with the boundary condition $u^{(s)} =0$ on $(-\infty, T) \times \{x_d =0\}$, where
\[
F^{(s)}(z) = \xi(t-s) F(z), \quad f^{(s)}(z)   = \xi(t-s) f(z)  +  \lambda^{-1/2}\xi'(t-s) u(z).
\]
As $\text{spt}(u^{(s)}) \subset (s -( R_0r_0)^2, s+ (R_0r_0)^2) \times \bR^{d}_{+}$, we can apply the estimate we just proved to infer that
\begin{align*}
& \|x_d^{\alpha}Du^{(s)}\|_{L_p(\Omega_T, d\mu_1)} + \sqrt{\lambda} \|x_d^{\alpha}u^{(s)}\|_{L_p(\Omega_T, d\mu_1)}\\
 &  \leq N \|x_d^{\alpha} F^{(s)}\|_{L_p(\Omega_T, d\mu_1)} +N\| x_d^{\alpha} f^{(s)}\|_{L_p(\Omega_T, d\mu_1)}.
\end{align*}
Raising to the $p$-th power and integrating this estimate with respect to $s$, we get
\begin{equation} \label{int-0515}
\begin{split}
& \int_{\bR}\Big( \|x_d^{\alpha} Du^{(s)}\|_{L_p(\Omega_T, d\mu_1)}^p + \lambda^{p/2} \|x_d^{\alpha} u^{(s)}\|^p_{L_p(\Omega_T, d\mu_1)}\Big)\, ds\\
&  \leq N\int_{\bR} \Big( \| x_d^{\alpha} F^{(s)}\|^p_{L_p(\Omega_T, d\mu_1)} + \|x_d^{\alpha} f^{(s)}\|^p_{L_p(\Omega_T, d\mu_1)} \Big)\, ds.
\end{split}
\end{equation}
It follows from the Fubini theorem and \eqref{xi-0515} that
\begin{align*}
\int_{\bR}\|x_d^{\alpha} Du^{(s)}\|_{L_p(\Omega_T,d\mu_1)}^p\, ds &= \int_{\Omega_T}\int_{\bR} |x_d^{\alpha} Du(z)|^p \xi^p(t-s)\, ds\,\mu_1(dz)\\
&= \|x_d^{\alpha} Du\|_{L_p(\Omega_T, d\mu_1)}^p.
\end{align*}
Similarly,
\begin{align*}
& \int_{\bR}\|x_d^{\alpha} u^{(s)}\|_{L_p(\Omega_T,d\mu_1)}^p\, ds = \|x_d^{\alpha} u\|_{L_p(\Omega_T, d\mu_1)}^p,  \\
 & \int_{\bR}\|x_d^{\alpha} F^{(s)}\|_{L_p(\Omega_T,d\mu_1)}^p\, ds = \|x_d^{\alpha} F\|_{L_p(\Omega_T, d\mu_1)}^p.
\end{align*}
Because $r_0$ depends only on $d$, $\alpha$, $\kappa$, and $p$, from the definition of $f^{(s)}$, \eqref{xi-0515}, and the Fubini theorem, we have
\[
\begin{split}
&\left(\int_{\bR} \|x_d^{\alpha} f^{(s)}\|_{L_p(\Omega, d\mu_1)}^p\, ds \right)^{1/p}\\ &\leq N  \|x_d^{\alpha} f\|_{L_p(\Omega_T, d\mu_1)} +  NR_0^{-2} \lambda^{-1/2}\|x_d^{\alpha} u\|_{L_p(\Omega_T, d\mu_1)}
\end{split}
\]
for $N = N(d,\alpha, \kappa, p)$.  Collecting the estimates that we have just derived, we infer from \eqref{int-0515} that
\[
\begin{split}
& \|x_d^{\alpha} Du\|_{L_p(\Omega_T, d\mu_1)} + \sqrt{\lambda} \|x_d^{\alpha} u\|_{L_p(\Omega_T, d\mu_1)}\\
&  \leq N\|x_d^{\alpha} F \|_{L_p(\Omega_T, d\mu_1)} +N\|x_d^{\alpha} f\| _{L_p(\Omega_T, d\mu_1)} + NR_0^{-2}\lambda^{-1/2}\|x_d^{\alpha} u\|_{L_p(\Omega_T, d\mu_1)}
\end{split}
\]
with $N=N(d,\alpha, \kappa, p)$. Now we choose $\lambda_0 = 2N$. For $\lambda \geq \lambda_0 R_0^{-2}$, we have $NR_0^{-2}\lambda^{-1/2} \leq \sqrt\lambda/2$, and therefore
\[
\begin{split}
& \|x_d^{\alpha} Du \|_{L_p(\Omega_T, d\mu_1)}
+ \sqrt{\lambda} \|x_d^{\alpha} u \|_{L_p(\Omega_T, d\mu_1)}\\
&  \leq   N \|x_d^{\alpha} F \|_{L_p(\Omega_T, d\mu_1)} +N \|x_d^{\alpha} f\| _{L_p(\Omega_T, d\mu_1)}+\frac{\sqrt{\lambda}}{2} \|x_d^{\alpha} u\|_{L_p(\Omega_T, d\mu_1)} ,
\end{split}
\]
which yields \eqref{eq3.23}.

Finally, the solvability of solution $u \in \sH_p^1(\Omega_T, x_d^{\alpha p} d\mu_1)$ can be obtained by the method of continuity using the solvability of the equation
\[
\left\{
\begin{aligned}
x_d^{\alpha}(u_t + \lambda u) - D_i(x_d^\alpha (D_i u - F_i)) & = \lambda^{1/2} x_d^\alpha f  \quad \text{in} \ \Omega_T, \\
u  & =  0\quad \text{on} \ \{x_d =0\}
\end{aligned} \right.
\]
in Theorem \ref{thm3}. The proof is now completed.
\end{proof}
\subsection{Proof of Corollary \ref{main-thrm}}
We now give the proof of Corollary \ref{main-thrm}.

\begin{proof}
We exploit an idea in \cite{KRW20}, which makes use of a duality argument.  Let $p_1 >p_0$ be such that
 \[
\left\{
\begin{aligned}
\frac{1}{p_0} \le  \frac{1}{d+2+ \alpha_-}+ \frac{1}{p_1} & \quad \text{if} \quad d \geq 2\\
\frac{1}{p_0} \le \frac{1}{4+ \alpha_-}+ \frac{1}{p_1} & \quad \text{if} \quad d =1.
\end{aligned} \right.
\]
Since $u \in \sH_{p_0}^1(Q_2^+, x_d^{\alpha p_0} d\mu_1)$, it follows from Lemma \ref{lem2.2} that
\begin{equation} \label{0831.emb}
u \in L_{p_1}(Q_2^+, x_d^{\alpha p_1}d\mu_1).
%\|x_{d}^{\alpha} u\|_{L_{p_1}(Q_2^+,d\mu_1)}\le N\big\||u|+|Du|+|F|+|f|\big\|_{L_{p_0}(Q_2^+, x_{d}^{\alpha \textcolor{red}{q}\textcolor{blue}{p_0}}d\mu_1)}.
\end{equation}
{\bf Case I}:  $p \leq p_1$. Without loss of generality, we may assume that $p^*\le p_0$ because otherwise we can replace $p^*$ with $p_0$ (noting that \eqref{eq3.19} and \eqref{eq3-2.19} still hold) and use H\"older's inequality.
Let $\eta\in C_0^\infty((-4,4)\times B_2)$ be such that $\eta\equiv 1$ on $Q_1$.  By a direct calculation, we see that $w = u\eta\in \sH^1_{p_0}(\Omega_0, x_d^{\alpha p_0}d\mu_1)$ is a weak solution of
\begin{equation} \label{eq6.16}
x_d^\alpha (w_t+\lambda w)- D_i\big(x_d^\alpha (a_{ij} D_j w - \widetilde F_i)\big)   =   x_d^\alpha \tilde f  \quad \text{in}\, (-4,0)\times \bR^d_+
\end{equation}
with the boundary condition $w =0$ on $(-4, 0) \times \partial \bR^d_+$ and the zero initial condition $w(-4,\cdot)=0$,
where
$$
\widetilde F_i=F_i\eta-a_{ij}u D_j\eta,
\quad \tilde f=f\eta+\lambda u\eta +u\eta_t-D_i\eta(a_{ij} D_ju-F_i),
$$
and $\lambda> \lambda_0 R_0^{-2}$ is a constant which will be chosen at the end.

Next, let $q=p/(p-1)$,  $q_0=p_0/(p_0-1)$, and $G=(G_1,\ldots,G_d) \in C_0^\infty(Q_1^+)^{d}$ and $g \in C_0^\infty(Q_1^+)$ satisfy
$$
\|G\|_{L_q(Q_1^+,d\mu_1)} = \|g\|_{L_q(Q_1^+,d\mu_1)} =1.
$$
By Theorem \ref{thm2}, there is a weak solution
 $v\in \sH^1_{q_0}((-4,0)\times \bR^d_+, x_d^{
\alpha q_0}d\mu_1)$ to
\begin{equation} \label{eq6.25}
-x_d^\alpha (v_t -\lambda v) - D_i\big(x_d^\alpha {a}_{ji} D_j v - G_i\big)  = \sqrt{\lambda} g \ \ \text{in}\ (-4,0)\times \bR^d_+
\end{equation}
with the boundary condition $v =0$ on $(-4, 0) \times \partial \bR^d_+$ and the zero terminal condition $v(0,\cdot)=0$. Since $q\le q_0$, and $G$ and $g$ are compactly supported, following the proof of Theorem \ref{thm3} \textup{\bf{(ii)}}, we have $v\in \sH^1_{q}((-4,0)\times \bR^d_+, x_d^{\alpha q}d\mu_1)$.
Moreover,
\begin{equation}
                                        \label{eq2.53}
\sqrt \lambda \|x_d^{\alpha}v\|_{L_q((-4,0)\times \bR^d_+,d\mu_1)}+\|x_d^{\alpha} Dv\|_{L_q((-4,0)\times \bR^d_+,d\mu_1)}\le N.
\end{equation}
Testing \eqref{eq6.16} and \eqref{eq6.25} with $v$ and $u\eta$ respectively, we get
\begin{align*}
& \int_{Q_1^+} \Big[ (x_d^{\alpha}D u)\cdot G + \sqrt{\lambda}(x_d^\alpha) g \Big]\,d\mu_1(z)\\
& =\int_{Q_2^+}\Big[ (x_d^\alpha D v)\cdot (x_d^\alpha  \widetilde F) + (x_d^\alpha v) (x_d^\alpha \tilde f) \Big]\,d\mu_1(z).
\end{align*}
Then, it follows from H\"older's inequality that
\begin{align} \notag
& \Big|\int_{Q_1^+}  \Big[ (x_d^\alpha D u) \cdot G + \sqrt{\lambda}(x_d^\alpha) g\Big]\,d\mu_1(z)\Big| \\ \label{eq7.12}
&\le \| x_d^\alpha Dv\|_{L_q(Q_2^+,d\mu_1)}\|x_d^\alpha \widetilde F\|_{L_p(Q_2^+,d\mu_1)}\\ \notag
& \qquad +\|x_d^\alpha v\|_{L_{q^*}(Q_2^+,d\mu_1)}\|x_d^\alpha \tilde f\|_{L_{p^*}(Q_2^+,d\mu_1)},
\end{align}
where $q^*=p^*/(p^*-1)$. From \eqref{eq6.25}, we see that $v\in \sH^1_{q}(Q_{2}^+, x_d^{\alpha q}d\mu_1)$ satisfies
\begin{equation*} %\label{eq6.52}
-x_d^\alpha v_t = D_i\big(x_d^\alpha {a}_{ji} D_j v \big) -D_i G_i  +  (-\lambda v x_d^{\alpha} + \sqrt{\lambda}g)
 \quad\text{in}\  Q_{2}^{+}.
\end{equation*}
When $\alpha\neq 0$, by \eqref{eq3.19}-\eqref{eq3-2.19}, $q^*$ satisfies the condition \eqref{eq2.06} in Lemma \ref{lem2.2}. Then by using Lemma \ref{lem2.2} and \eqref{eq2.53}, we get
\begin{align}
                                \label{eq3.07}
&\|x_d^\alpha v\|_{L_{q^*}(Q_2^+,d\mu_1)}\notag\\
&\le N\|x_d^\alpha  v\|_{L_{q}(Q_{2}^+,d\mu_1)}+ N\|x_d^\alpha  Dv\|_{L_{q}(Q_{2}^+,d\mu_1)} +  N\|x_d^\alpha  v_t\|_{\bH_q^{-1}(Q_{2}^+, d\mu_1)}   \notag\\
&\leq N+ N\|G\|_{L_{q}(Q_{2}^+,d\mu_1)}+N\|-\lambda v x_d^{\alpha} + \sqrt{\lambda}g\|_{L_{q}(Q_{2}^+,d\mu_1)}
\le N\sqrt\lambda.
\end{align}
When $\alpha=0$, by the usual unweighted parabolic Sobolev embedding, we still get \eqref{eq3.07}. It then follows from \eqref{eq7.12}, \eqref{eq2.53}, \eqref{eq3.07}, and the arbitrariness of $G$ and $g$ that
\begin{align}
                    \label{eq2.57}
&\|x_d^{\alpha}Du\|_{L_p(Q_1^+,d\mu_1)} + \sqrt{\lambda} \|x_d^{\alpha}u\|_{L_p(Q_1^+,d\mu_1)} \notag\\
&\le N\|x_d^{\alpha} \widetilde F\|_{L_p(Q_2^+,d\mu_1)}+N \sqrt\lambda \|x_d^{\alpha} \tilde f\|_{L_{p^*}(Q_2^+,d\mu_1)}\notag\\
&\le N(\sqrt\lambda+1)\|x_d^{\alpha} F\|_{L_p(Q_2^+,d\mu_1)}+N\|x_d^{\alpha} u\|_{L_p(Q_2^+,d\mu_1)}
+N \sqrt\lambda\| x_d^{\alpha}f\|_{L_{p^*}(Q_2^+,d\mu_1)}\notag\\
&\quad +N  \sqrt \lambda (\lambda +1)\|x_d^{\alpha} u\|_{L_{p^*}(Q_2^+,d\mu_1)}
+N \sqrt\lambda \|x_d^{\alpha} Du\|_{L_{p^*}(Q_2^+,d\mu_1)},
\end{align}
where $N$ is independent of $\lambda$. Observe that by the assumptions in the corollary, \eqref{0831.emb}, $p \leq p_1$, and $p^* \leq p_0$, all the terms on the right-hand side of \eqref{eq2.57} are finite. Note also that from  H\"{o}lder's inequality, it follows that
\begin{align*}
& \|u\|_{L_1(Q_2^+)} =\|x_d^\alpha u\|_{L_{1}(Q_1^+, d\mu_1)} \leq N(p_0, d) \|x_d^\alpha u\|_{L_{p_0}(Q_1^+, d\mu_1)} <\infty, \\
& \|Du\|_{L_1(Q_2^+)} =\|x_d^\alpha Du\|_{L_{1}(Q_1^+, d\mu_1)} \leq N(p_0, d) \|x_d^\alpha Du\|_{L_{p_0}(Q_1^+, d\mu_1)} <\infty.
\end{align*}
Therefore, as $p^*<p$, we conclude \eqref{main-thm-estb} from \eqref{eq2.57} by using H\"older's inequality and a standard iteration argument with the underlying measure $d\mu_1$ and for a sufficiently large $\lambda$.  See, for example, \cite[pp. 80--82]{Giaq}. The corollary is proved when $p \leq p_1$.
\\
\noindent
{\bf Case II}: $p > p_1$. By applying the result in {\bf Step I}, we obtain  \eqref{main-thm-estb} with $p_1$ in place of $p$.  From this, we can use the argument in {\bf Step I} again with $p_1$ in place of $p_0$. After iterating the argument for a finite number of steps, we  obtain \eqref{eq2.57} for general $p$.
\end{proof}

\subsection{Proof of Theorem \ref{thm1.4}}

In order to prove Theorem \ref{thm1.4}, we need the following higher regularity estimates of weak solutions to the homogeneous equation \eqref{eq11.52}-\eqref{eq12.01}. Recall that the H\"{o}lder semi-norm $C^{1/2, 1}$ is defined as
\[
[f]_{C^{1/2, 1}(Q_{1/2}^+)} = \sup_{\substack{(t,x), (s,y) \in Q_{1/2}^+ \\ (t,x) \not= (s,y)}}  \frac{|f(t,x) - f(s, y)|}{|t-s|^{1/2} + |x-y|}.
\]

\begin{corollary}[Higher regularity]
                        \label{cor5.2}
Under the assumptions of Proposition \ref{lem1}, for any $q\in [1,2]$ we have
\begin{equation}
                            \label{eq11.27}
[x_d^{\alpha}u]_{C^{1/2,1}(Q_{1/2}^+)}
\le N \Big(\fint_{Q^+_{1}}|\tilde{x}_d^{\alpha} u (\tilde{z})|^q \, \mu_1(d\tilde{z})\Big)^{1/q},
\end{equation}
\begin{equation}
                            \label{eq11.28}
[x_d^\alpha D_{x'}u]_{C^{1/2,1}(Q_{1/2}^+)}\le N \Big(\fint_{Q^+_{1}}|\tilde{x}_d^\alpha D_{x'}u(\tilde{z})|^q\, \mu_1(d\tilde{z})\Big)^{1/q},
\end{equation}
and
\begin{equation}
                            \label{eq11.29}
[x_d^\alpha \cU]_{C^{1/2,1}(Q_{1/2}^+)}\le N \Big(\fint_{Q^+_{1}}\big(|\tilde{x}_d^\alpha Du(\tilde{z})|  + \lambda^{1/2} |\tilde{x}_d^\alpha u(\tilde{z})| \big)^q\, \mu_1(d\tilde{z})\Big)^{1/q},
\end{equation}
where $\cU$ is defined in \eqref{eq11.31} and $N=N(d,\alpha,\kappa,q)>0$.
\end{corollary}
\begin{proof}
We first consider the case when $q=2$.
Since $u_t$ satisfies the same equation as $u$, from \eqref{eq5.23}, \eqref{inter-eq3.50}, and \eqref{eq11.58}, we have for any $z\in Q_{1/2}^+$,
\begin{align}
                    \label{eq12.48}
&|\partial_t(x_d^\alpha u)|=|x_d^\alpha u_t|\le N \Big(\fint_{Q^+_{3/4}} |\tilde{x}_d^\alpha u_t(\tilde{z})|^2 \, \mu_1(d\tilde{z})\Big)^{1/2}\notag\\
& \le N\Big(\fint_{Q^+_{7/8}}\big(|\tilde{x}_d^\alpha Du(\tilde{z})|^2  + \lambda |\tilde{x}_d^\alpha u(\tilde{z})|^2 \big) \, \mu_1(d\tilde{z})\Big)^{1/2}\\
& \le N\Big(\fint_{Q^+_{1}}|\tilde{x}_d^{\alpha} u (\tilde{z})|^2 \, \mu_1(d\tilde{z})\Big)^{1/2}\notag.
\end{align}
Next, it follows from  \eqref{eq5.23}, \eqref{eq5.24}, and \eqref{eq11.58} that  for any $z\in Q_{1/2}^+$,
\begin{align*}
&|D_{x'}(x_d^\alpha u)|=|x_d^\alpha D_{x'}u|\le N \Big(\fint_{Q^+_{3/4}}\big(|\tilde{x}_d^\alpha Du(\tilde{z})|^2  + \lambda |\tilde{x}_d^\alpha u(\tilde{z})|^2 \big)\, \mu_1(d\tilde{z})\Big)^{1/2}\\
& \le N\Big(\fint_{Q^+_{1}}|\tilde{x}_d^{\alpha} u (\tilde{z})|^2 \, \mu_1(d\tilde{z})\Big)^{1/2}
\end{align*}
and similarly
\begin{align*}
|D_{d}(x_d^\alpha u)|\le |x_d^\alpha D_{d}u|+|\alpha x_d^{\alpha-1} u|
%&\le N \Big(\fint_{Q^+_{3/4}}\big(|\tilde{x}_d^\alpha Du(\tilde{z})|^2  + \lambda |\tilde{x}_d^\alpha u(\tilde{z})|^2 \big)\, \mu_1(d\tilde{z})\Big)^{1/2}\\
 \le N\Big(\fint_{Q^+_{1}}|\tilde{x}_d^{\alpha} u (\tilde{z})|^2 \, \mu_1(d\tilde{z})\Big)^{1/2}.
\end{align*}
Combining the estimates above, we obtain \eqref{eq11.27}. Because $D_{x'}u$ satisfies the same equation as $u$, we get \eqref{eq11.28} immediately from \eqref{eq11.27}.

Next we prove \eqref{eq11.29}. Since $u_t$ and $D_{x'}u$ satisfy the same equation as $u$, by using \eqref{eq5.24} and \eqref{inter-eq3.50}, we get for any $z\in Q_{1/2}^+$,
\begin{align*}
&|\partial_t(x_d^\alpha \cU)|=|x_d^\alpha  \overline{a}_{dj}(x_d)D_j u_t|\\
&\le N \Big(\fint_{Q^+_{3/4}} \big(|\tilde{x}_d^\alpha Du_t(\tilde{z})|^2  + \lambda |\tilde{x}_d^\alpha u_t(\tilde{z})|^2 \big)\, \mu_1(d\tilde{z})\Big)^{1/2}\\
& \le N\Big(\fint_{Q^+_{1}}\big(|\tilde{x}_d^\alpha Du(\tilde{z})|^2  + \lambda |\tilde{x}_d^\alpha u(\tilde{z})|^2 \big) \, \mu_1(d\tilde{z})\Big)^{1/2}
\end{align*}
and
\begin{align*}
&|D_{x'}(x_d^\alpha \cU)|=|x_d^\alpha  \overline{a}_{dj}(x_d)D_j D_{x'}u|\\
&\le N \Big(\fint_{Q^+_{3/4}} \big(|\tilde{x}_d^\alpha DD_{x'}u(\tilde{z})|^2  + \lambda |\tilde{x}_d^\alpha D_{x'}u(\tilde{z})|^2 \big)\, \mu_1(d\tilde{z})\Big)^{1/2}\\
& \le N\Big(\fint_{Q^+_{1}}\big(|\tilde{x}_d^\alpha Du(\tilde{z})|^2  + \lambda |\tilde{x}_d^\alpha u(\tilde{z})|^2 \big) \, \mu_1(d\tilde{z})\Big)^{1/2}.
\end{align*}
By using \eqref{eq12.46}, \eqref{eq12.48}, \eqref{eq11.28}, and \eqref{eq5.23}, we obtain
\begin{align*}
&|D_{d}(x_d^\alpha \cU)|\le Nx_d^\alpha(|u_t|+|DD_{x'}u|+\lambda|u|)\\
%&\le N \Big(\fint_{Q^+_{3/4}} \big(|\tilde{x}_d^\alpha DD_{x'}u(\tilde{z})|^2  + \lambda |\tilde{x}_d^\alpha D_{x'}u(\tilde{z})|^2 \big)\, \mu_1(d\tilde{z})\Big)^{1/2}\\
& \le N\Big(\fint_{Q^+_{1}}\big(|\tilde{x}_d^\alpha Du(\tilde{z})|^2  + \lambda |\tilde{x}_d^\alpha u(\tilde{z})|^2 \big) \, \mu_1(d\tilde{z})\Big)^{1/2}.
\end{align*}
Combining the above three estimates, we reach \eqref{eq11.29}.

Finally, when $q \in [1,2)$, we use the result for $q=2$, Proposition \ref{lem1}, and a standard iteration. See, for example, \cite[pp. 80--82]{Giaq}.
The corollary is proved.
\end{proof}

\begin{proof}[Proof of Theorem \ref{thm1.4}]   From Corollary \ref{cor5.2}, we can apply method of mean oscillation estimates introduced in \cite{Krylov}. As this is similar to that of \cite[Theorem 2.4]{DP20}, we skip the details and only outline some important steps in Appendix \ref{pr-thm1.4}.
\end{proof}
\appendix
\section{Proof of Theorem \ref{thm3}}  \label{proof-thm4.6}
\begin{proof}
We first prove Assertion {\bf (i)}. Let $u \in \sH_{q}^1(\Omega_T, x_d^{\alpha q}d\mu_{1})$ be a weak solution of  \eqref{eq11.52a}-\eqref{eq12.01a}.  By H\"{o}lder's inequality and as $\alpha <1$ and $q \geq 2$, we have $u \in \sH_{2, \textup{loc}}^1(\Omega_T, d\mu)$. Then, from Proposition \ref{Simple-approx}, it follows that for every $z_0 \in \overline{\Omega}_T$ and $r >0$, we have the decomposition
\[
u =v+ w \quad \text{in} \  Q_{10r}^+(z_0),
\]
where $v$ and $w$ satisfy \eqref{0506-tU-est} and \eqref{0506-W.est}. From this, we obtain \eqref{thm3-est} by using  the real variable argument. See, for instance, \cite{DK11b} and \cite[Lemma A.20]{DK19}. As this is by now standard, we skip the details.

Next, we prove Assertion {\bf (ii)}. We split the proof into two cases when $p \in (2,\infty)$ and when $p \in (1,2)$.

\noindent {\bf Case I}: $p \in (2,\infty)$.  We only need to prove the existence of the solution, as the uniqueness  follows from \eqref{thm3-est} in {\bf(i)}. For $k=1,2,\ldots$, let $F^{(k)}=F(z) \chi_{\widehat Q_k}(z)$, where  $\widehat{Q}_k(z)$ is defined in \eqref{eq3.39}. It is clear that $x_d^\alpha F^{(k)}\in L_2(\Omega_T,d\mu_1)^{d}\cap L_p(\Omega_T,d\mu_1)^{d}$ by H\"older's inequality and moreover $x_d^{\alpha}F^{(k)}\to x_d^{\alpha}F$ in $L_p(\Omega_T, d\mu_1)$ as $k\to \infty$ by the dominated convergence theorem. Similarly, we find $\{x_d^\alpha f^{(k)}\}\subset L_2(\Omega_T,d\mu_1)\cap L_p(\Omega_T,d\mu_1)$.  Now, let $u^{(k)} {\in \sH_2^1(\Omega_T, d\mu)}$ be the weak solution of the equation \eqref{eq11.52a}-\eqref{eq12.01a} with $F^{(k)}$ and $ f^{(k)}$ in place of $F$ and $f$, respectively. The existence of $u^{(k)}$ follows from Lemma \ref{L-2-lemma}. Also, from Assertion {\bf(i)}, we have $u^{(k)}\in \sH^{1}_p(\Omega_T, x_d^{\alpha p} d\mu_1)$. By the strong convergence of $\{x_d^{\alpha}F^{(k)}\}$ and $\{x_d^{\alpha} f^{(k)}\}$ in $L_{p}(\Omega_T,d\mu_1)$, we infer that
$\{u^{(k)}\}$ is a Cauchy sequence in $\sH^{1}_{p}(\Omega_T, x_d^{\alpha p} d\mu_1)$. Let $u\in \sH^{1}_{p}(\Omega_T, x_d^{\alpha p}d\mu_1)$ be its limit. Then, by passing to the limit in the weak formulation of solutions, we see that $u$ is a solution to the equation \eqref{eq11.52a}-\eqref{eq12.01a}.

\noindent
{\bf Case II}: $p \in (1, 2)$. We use the method of duality. Although similar ideas are used in \cite[Theorem 4.1]{DP20},  the proof contains different details, which we give here. We first prove the estimate \eqref{thm3-est}. Let $q = {p}/(p-1) \in (2,\infty)$ and let $G:\Omega_T \rightarrow \mathbb{R}^d$ and  $g: \Omega_T \rightarrow \mathbb{R}$ be measurable functions such that  $|G| +|g| \in L_q(\Omega_T, d\mu_1)$. We consider the adjoint problem in $\bR \times \bR^d_+$
\begin{equation} \label{adj-eqn-sim}
 x_d^\alpha(- \bar a_0 v_t + \lambda \bar c_0 v) - D_i\big(x_d^\alpha(\overline{a}_{ji}(x_d) D_{j} v - \tilde{G}_i)\big)   =  \lambda^{1/2} x_d^\alpha \tilde{g}
\end{equation}
in $\bR \times \bR_+^d$ with the boundary condition
\begin{equation}  \label{v-0721-eqn}
v =0 \quad \text{on} \ \bR \times \partial \bR^d_+,
\end{equation}
where
\[
\tilde{G}(z) = x_d^{-\alpha}G(z) \chi_{(-\infty, T)}(t), \quad \tilde{g}(z) =   x_d^{-\alpha} g(z) \chi_{(-\infty, T)}(t).
\]
Observe that
\[
\begin{split}
&\|x_d^\alpha \tilde{G}\|_{L_q(\bR \times \bR^{d}_+, d\mu_1)} = \|G\|_{L_q(\Omega_{T}, d\mu_1)} \quad \text{and} \\
& \|x_d^\alpha \tilde{g}\|_{L_q(\bR \times \bR^{d}_+, d\mu_1)} = \|g\|_{L_q(\Omega_{T}, d\mu_1)}.
\end{split}
\]
By {\bf Case I}, there is a unique solution $v \in \sH^1_q(\bR \times \bR_+^d,  x_d^{\alpha q}d\mu_1)$ to \eqref{adj-eqn-sim}-\eqref{v-0721-eqn}, which satisfies
\begin{equation}
                            \label{eq10.35}
                            \begin{split}
& \int_{\bR \times \bR^d_+} \big(|x_d^{\alpha}Dv|^q + \lambda^{q/2} |x_d^{\alpha} v|^q \big) \,\mu_1(dz) \\
& \leq N\int_{\Omega_T} \big(|G|^q + |g|^q \big) \,\mu_1(dz).
\end{split}
\end{equation}
Also, by the uniqueness of solutions, we have $v =0$ for $t \geq T$. Then,  by testing \eqref{eq11.52a} with $v$, and testing \eqref{adj-eqn-sim} with $u$, and by using the definitions of $\tilde{G}$ and $\tilde{g}$, we obtain
\[
\begin{split}
& \int_{\Omega_T}\big[ G \cdot (x_d^{\alpha}  D u)
+ \lambda^{1/2}  g (x_d^{\alpha} u) \big]\,\mu_1(dz) \\
& = \int_{\Omega_T}\big[(x_d^{\alpha} F)\cdot (x_d^{\alpha} D v) + \lambda^{1/2} (x_d^{\alpha}f) (x_d^{\alpha}v) \big]\,\mu_1(dz).
\end{split}
\]
Now,  it follows from H\"older's inequality and \eqref{eq10.35} that
\[
\begin{split}
& \left|\int_{\Omega_T}\big(G\cdot (x_d^{\alpha} D u)
+ \lambda^{1/2} g (x_d^\alpha u) \big)\,\mu_1(dz)\right| \\
& \leq  \|x_d^{\alpha} F\|_{L_p(\Omega,d\mu_1)} \|x_d^{\alpha}  D v\|_{L_q(\Omega_T, d\mu_1)} + \lambda^{1/2} \|x_d^{\alpha} f\|_{L_{p}(\Omega_T, d\mu_1)} \|x_d^{\alpha}v\|_{L_q(\Omega_T, d\mu_1)}\\
& \leq N\Big(\|x_d^{\alpha} F\|_{L_p(\Omega,d\mu_1)} + \|x_d^{\alpha} f\|_{L_{p}(\Omega_T, d\mu_1)}\Big)
\Big( \|G\|_{L_q(\Omega_T, d\mu_1)} + \|g\|_{L_q(\Omega_T, d\mu_1)} \Big).
\end{split}
\]
From the last estimate and as $G$ and $g$ are arbitrary, we obtain \eqref{thm3-est}.

It now remains to prove the existence of solution $u \in \sH_p^1(\Omega_T, x_d^{\alpha p} d\mu_1)$.  For $i=1,2,\ldots, d$ and $k=1,2,\ldots$, let
$$
F^{(k)}_i (z) =x_d^{-\alpha}\max\big \{-k,\min\{k, x_d^{\alpha}F_i(z) \}\big \}\chi_{\widehat Q_k}(z),
$$
where $\widehat Q_{r}$ is defined in \eqref{eq3.39}. Then $x_d^{\alpha}F^{(k)}\in L_2(\Omega_T,d\mu_1)^{d}\cap L_p(\Omega_T, d\mu_1)^{d}$ and by the dominated convergence theorem, $x_d^{\alpha}F^{(k)}\to x_d^{\alpha}F$ in $L_p(\Omega_T,d\mu_1)$ as $k\to \infty$. Similarly, we find $\{x_d^{\alpha}f^{(k)}\}\subset L_2(\Omega_T,d\mu_1)\cap L_p(\Omega_T,d\mu_1)$. By Lemma \ref{L-2-lemma}, there is a unique weak solution $u^{(k)}\in \cH^{1}_{2}(\Omega_T,d\mu)$ to the equation \eqref{eq11.52a}-\eqref{eq12.01a} with $F^{(k)}$ and $f^{(k)}$ in place of $F$ and $f$, respectively.

As in {\bf Case I},  we only need to prove that  $u^{(k)} \in \sH_p^1(\Omega_T,  x_d^{\alpha p}d\mu_1)$.  However, the proof of this is more involved because we cannot apply Assertion (i) as before. We adapt the idea in \cite[Section 8]{MR3812104} by using a localization and partition argument.  Let us fix a $k \in \mathbb{N}$. As ${\mu_1}$ is a doubling measure, there exists $N_0 = N_0(\alpha, d)>0$  such that
 \begin{equation}  \label{eq7.36}
{\mu}_1(\widehat Q_{2r})\le N_0{\mu}_1(\widehat Q_{r}), \quad \forall \ r >0.
\end{equation}
Since $u^{(k)}\in \sH^{1}_{2}(\Omega_T,d\mu)$ and $p \in (1,2)$, we apply H\"older's inequality and see that
\begin{equation} \label{eq7.38}
 \|x_d^{\alpha}u^{(k)}\|_{L_p(\widehat Q_{2k},d\mu_1)}+\|x_d^{\alpha} Du^{(k)}\|_{L_p(\widehat Q_{2 k} ,d\mu_1)}<\infty.
\end{equation}
Hence, we only need to prove that
$$
\|x_d^{\alpha}u^{(k)}|\|_{L_p(\Omega_T\setminus \widehat{Q}_{2k}, d\mu_1)} +  \|x_d^{\alpha}Du^{(k)}|\|_{L_p(\Omega_T\setminus \widehat{Q}_{2k}, d\mu_1)} < \infty.
$$
For each $l  \geq 0$,  let $\eta_{l}$ be a smooth function such that
\begin{align*}
\eta_l&\equiv 0\quad \text{in}\ \widehat Q_{2^lk},  \quad \eta_l \equiv 1 \quad \text{outside}\,\, \widehat Q_{2^{l+1}k},
\end{align*}
and $|D \eta_l|\le C_0 2^{-l}$, $|(\eta_l)_t|\le C_0 2^{-2l}$, where $C_0$ is independent of $l$. Let us also denote $w^{(k,l)}=u^{(k)}\eta_l$. We see that $w^{(k,l)}\in \sH^{1}_{2}(\Omega_T,d\mu)$ is a weak solution of
\begin{equation*}
x_d^\alpha \big( \overline{a}_0 w^{(k,l)}_t  +\lambda  \overline{c}_0 w^{(k,l)}\big) - D_i\big(x_d^\alpha (\overline{a}_{ij} D_j w^{(k,l)} - F^{(k,l)}_i)\big)   =  \lambda^{1/2} x_d^\alpha f^{(k,l)}
\end{equation*}
in $\Omega_T$ with the boundary condition $w^{(k,l)} = 0$ on $(-\infty, T) \times\{x_d =0\}$, where
\[
\begin{split}
F^{(k,l)}_i & =  u^{(k)}\overline{a}_{ij}D_j \eta_l, \quad i = 1, 2,\ldots, d,\\
 f^{(k,l)} & = \lambda^{-1/2}\big(\overline{a}_0u^{(k)}(\eta_l)_t - \overline{a}_{ij} D_j u^{(k)} D_i\eta_l \big).
\end{split}
\]
Here we used $\eta_l F_i^{(k)} \equiv \eta_l f^{(k)}  \equiv F_i^{(k)} D_i\eta_l \equiv 0$ for every $i = 1, 2,\ldots, d$ and $l =0,1,\ldots$.

Now, applying the estimate \eqref{L2-lemma-est} to the equation of $w^{(k,l)}$, we have
\begin{align*}
& \big\|x_d^{\alpha}(|Dw^{(k,l)}|+\lambda^{1/2}|w^{(k,l)}|) \big\|_{L_2(\Omega_T,d\mu_1)}\\%+ \sqrt{\lambda} \|x_d^{\alpha} w^{(k,l)}\|_{L_2(\Omega_T,d\mu_1)}\\
& \le N \|x_d^{\alpha} F^{(k,j)}\|_{L_2(\Omega_T,d\mu_1)}+N
\|x_d^{\alpha} f^{(k,l)}\|_{L_2(\Omega_T,d\mu_1)}.
\end{align*}
This implies that
\begin{align*}
&\big\|x_d^{\alpha} (|Du^{(k)}|+\lambda^{1/2}|u^{(k)}|)\big\|_{L_2(\widehat Q_{2^{l+2}k}\setminus \widehat Q_{2^{l+1}k},d\mu_1)}\\%+\sqrt{\lambda} \|x_d^{\alpha} u^{(k)}\|_{L_2(\widehat Q_{2^{j+2}k}\setminus \widehat Q_{2^{j+1}k},d\mu_1)}\\
&\le
N\Big( 2^{-l}\|x_d^{\alpha} u^{(k)}\|_{L_2(\widehat Q_{2^{l+1}k}\setminus \widehat Q_{2^{l}k},d\mu_1)}+ \lambda^{-1/2}2^{-2l}
\|x_d^{\alpha} u^{(k)}\|_{L_2(\widehat Q_{2^{l+1}k}\setminus \widehat Q_{2^{l}k},d\mu_1)}\\
&\quad + \lambda^{-1/2}2^{-l}\|x_d^{\alpha} Du^{(k)}\|_{L_2(\widehat Q_{2^{l+1}k}\setminus \widehat Q_{2^{l}k},d\mu_1)} \Big)\\
&\le C2^{-l}\big\|x_d^{\alpha} (|Du^{(k)}|+\lambda^{1/2}|u^{(k)}|)\big\|_{L_2(\widehat Q_{2^{l+1}k}\setminus \widehat Q_{2^{l}k},d\mu_1)}
\end{align*}
for every $l \geq 1$, where $C$ also depends on $\lambda$, but is independent of $l$. Then, by iterating this estimate, we obtain
\begin{align}  \label{eq9.01}
  &\big\|x_d^{\alpha} (|Du^{(k)}|+\lambda^{1/2}|u^{(k)}|)\big\|_{L_2(\widehat Q_{2^{l+1}k}\setminus \widehat Q_{2^{l}k},d\mu_1)}\nonumber\\
&\le C^l2^{-\frac{l(l-1)}2}\big\|x_d^{\alpha}(| Du^{(k)}|+ \lambda^{1/2}|u^{(k)}|)\big\|_{L_2(\widehat Q_{2k},d\mu_1)}.
\end{align}
As $p \in (1, 2)$, we apply H\"older's inequality, \eqref{eq7.36}, and \eqref{eq9.01} to obtain
\begin{align*}
&\big\|x_d^{\alpha} (|Du^{(k)}|+\lambda^{1/2}|u^{(k)}|)\big\|_{L_p(\widehat Q_{2^{l+1}k}\setminus \widehat Q_{2^{l}k},d\mu_1)}\\
&\le ({\mu_1}(\widehat Q_{2^{l+1}k}))^{\frac 1 p-\frac 1 2}\big\|x_d^{\alpha} (|Du^{(k)}|+\lambda^{1/2}|u^{(k)}|)\big\|_{L_2(\widehat Q_{2^{l+1}k}\setminus \widehat Q_{2^{l}k},d\mu_1)}\\
&\le (N_0^l{\mu_1}(\widehat Q_{2k}))^{\frac{1}{p}-\frac 1 2}C^l2^{-\frac {l(l-1)} 2}\big\|x_d^{\alpha}(| Du^{(k)}|+\lambda^{1/2}|u^{(k)}|)\big\|_{L_2(\widehat Q_{2k},d\mu_1)}.
\end{align*}
Then, it follows that
\[
\begin{split}
&\big\|x_d^{\alpha} (|Du^{(k)}|+\lambda^{1/2}|u^{(k)}|)\big\|_{L_p(\Omega_T \setminus \widehat Q_{2k},d\mu_1)}
\\
&\le \sum_{l=1}^\infty \big\|x_d^{\alpha} (|Du^{(k)}|+\lambda^{1/2} |u^{(k)}|)\big\|_{L_p(\widehat Q_{2^{l+1}k}\setminus \widehat Q_{2^{l}k},d\mu_1)}\\
& \leq N \big\|x_d^{\alpha} (|Du^{(k)}|+\lambda^{1/2}|u^{(k)}|)\big\|_{L_2(\widehat Q_{2k},d\mu_1)} < \infty.
\end{split}
\]
From this estimate and \eqref{eq7.38}, we infer that $u^{(k)} \in \sH_p^1(\Omega_T, x_{d}^{\alpha p}d\mu_1)$. The theorem is proved.
\end{proof}
%====
\section{Sketch of the proof of Theorem \ref{thm1.4}} \label{pr-thm1.4}
\begin{proof} It follows from Corollaries \ref{main-thrm} and \ref{cor5.2} as well as the corresponding interior estimate that for any $q_0\in (1,2)$, if $v \in {\sH}^1_{q_0}(Q^+_r(z_0),  x_d^{\alpha q_0}d\mu_1)$ is a weak solution of \eqref{eq11.52}-\eqref{eq12.01}, then we have
\begin{equation}
                                        \label{eq5.52b}
\begin{split}
& [x_d^\alpha D_{x'}v]_{C^{1/2,1}(Q_{r/2}^+(z_0))}
+[x_d^\alpha\mathcal{V}]_{C^{1/2,1}(Q_{r/2}^+(z_0))}
+\sqrt \lambda [x_d^\alpha v]_{C^{1/2,1}(Q_{r/2}^+(z_0))} \\
& \le Nr^{-1}\Big(\fint_{Q^+_{r}(z_0)}|x_d^\alpha Dv|^{q_0}
+\lambda^{q_0/2}|x_d^\alpha v|^{q_0}\,\mu_1(dz)\Big)^{1/q_0},
\end{split}
\end{equation}
where $\mathcal{V} =\overline{a}_{dj}(x_d) D_jv$. By using \eqref{eq5.52b}, Theorem \ref{thm2}, and a decomposition argument as in the proof of Proposition  \ref{G-approx-propos}, we have the following the mean oscillation estimate:  if $\textup{spt}(u) \subset  (s -( R_0r_0)^2, s+ (R_0r_0)^2) \times \bR^{d}_+$ for some $s\in \bR$, then for any $\tau\le 1/30$ and $z_0\in \overline{\Omega_T}$,
\begin{align*}
&\fint_{Q_{\tau r}^+(z_0)}\big[|x_d^\alpha D_{x'}u-(x_d^\alpha D_{x'}u)_{Q_{\tau r}^+(z_0)}|\\
&\quad +|x_d^\alpha \cU-(x_d^\alpha \cU)_{Q_{\tau r}^+(z_0)}|
+\sqrt\lambda|x_d^\alpha u-(x_d^\alpha u)_{Q_{\tau r}^+(z_0)}|\big]\,{\mu_1}(dz)\\
&\leq  N  \tau^{-(d+2+\alpha_-)}r_0^{2(1-\frac 1 {q_0})}
\Big(\fint_{Q_{r}^+(z_0)}|x_d^\alpha Du|^{q_0} \,{\mu_1}(dz)\Big)^{\frac 1 {q_0}}\\
&\ +N  \tau^{-\frac {d+2+\alpha_-} {q_0}}\Big(\fint_{Q_{r}^+(z_0)}( |x_d^\alpha F|^{q_0} +  |x_d^\alpha f|^{q_0}) \,{\mu_1}(dz)\Big)^{\frac 1 {q_0}}\\
&\ +N \tau \Big(\fint_{Q^+_{r}(z_0)}|x_d^\alpha Du|^{q_0}
+\lambda^{q_0/2}|x_d^\alpha u|^{q_0}\,\mu_1(dz)\Big)^{\frac 1{q_0}}\\
&\ +N \tau^{-\frac {d+2+\alpha_-} {q_0}}\delta_0^{\frac 1 {q_0\nu_1}} \Big(\fint_{Q^+_{r}(z_0)}|x_d^\alpha Du|^{q_0\nu_2}\,\mu_1(dz)\Big)^{\frac 1 {q_0\nu_2}},
\end{align*}
where $\nu_1\in (1,\infty)$, $\nu_2=\nu_1/(\nu_1-1)$, and $\cU = a_{dj} D_j u$.
The a priori estimate \eqref{main-thm-estc} then follows from the mean oscillation estimate, the reverse H\"older's inequality for $A_p$ weights, the weighted and mixed-norm Fefferman--Stein type theorems on sharp functions, and the weighted and mixed-norm Hardy--Littlewood maximal function theorem. See, for instance, Corollary 2.6, 2.7, and  Section 7 of \cite{MR3812104} for details. The solvability in weighted and mixed-norm Sobolev spaces then follows from the estimate \eqref{main-thm-estc} and an approximation argument by using the solvability result in Theorem \ref{thm2}. We omit the details and refer the reader to \cite[Section 8]{MR3812104}.
\end{proof}
\section{Proofs of Remark \ref{rem2.7} and Corollary \ref{ell-local-th}} \label{App-C}
%=====
For completeness, we provide the proofs of Remark \ref{rem2.7} and Corollary \ref{ell-local-th}. We need the following embedding result that is slightly more general than Lemma \ref{lem2.2}.
\begin{lemma}
        \label{imbedding-two}
Let $\alpha, \tilde{\alpha}\in (-\infty, 1)$ and $q \in (1, \infty)$ be fixed numbers such that $\tilde{\alpha} -1>q(\alpha -1)$.  Let $q^*\in (1,\infty)$ satisfy
\begin{equation}
                        \label{0828.alpha}
\left\{
\begin{aligned}
\frac{1}{q} \le  \frac{1}{d+2+ \tilde{\alpha}_-}+ \frac{1}{q^*} & \quad \text{if} \quad d \geq 2\\
 \ \frac{1}{q}  \le  \ \frac{1}{4+ \tilde{\alpha}_-}+ \frac{1}{q^*} & \quad \text{if} \quad d =1.
\end{aligned} \right.
\end{equation}
Then for any $u\in \sH^1_q(Q_2^+, x_{d}^{\alpha q}d\tilde{\mu})$, we have
\begin{equation}
                    \label{eq0828}
\|u\|_{L_{q^*}(Q_2^+, x_{d}^{\alpha q^*} d\tilde{\mu})}\le N\|u\|_{\sH^1_{q}(Q_2^+, x_{d}^{\alpha q}d\tilde{\mu})},
\end{equation}
where $N=N(d,\alpha, \tilde{\alpha},q,q^*)>0$ is a constant, $\tilde{\alpha}_- =\max\{-\tilde{\alpha}, 0\}$, and $\tilde{\mu}(dz) = x_d^{-\tilde{\alpha}} dx dt$. The result still holds when $q^*=\infty$ and the inequalities in \eqref{0828.alpha} are strict.
\end{lemma}
\begin{proof} Let us define $w = x_d^{\alpha} u$. We have
\[
D_i w = x_d^\alpha D_i u + \alpha x_d^{\alpha -1} \delta_{id}u, \quad i = 1, 2,\ldots, d,
\]
where $\delta_{id} =1$ when $i =d$ and $\delta_{id} =0$ otherwise. By the fundamental theorem of calculus, we have
\[
|x_d^{\alpha -1} u(z', x_d)| \leq x_d^\alpha \int_0^1 |D_d u(z', s x_d)| ds.
\]
Then, by applying the Minkowski inequality, we obtain
\[
\begin{split}
\|x_d^{\alpha -1} u\|_{L_q(Q_2^+, d\tilde{\mu})} & \leq \int_0^1\Big(\int_{Q_2^+} |D_du(z', sx_d)|^q x_d^{\alpha q - \tilde{\alpha}} dz' dx_d \Big)^{\frac{1}{q}} ds \\
& = \|x_d^\alpha D_d u\|_{L_q(Q_2^+, d\tilde{\mu})} \int_0^1 s^{(\tilde{\alpha} -1)/q - \alpha} ds \\
& = N \|x_d^\alpha D_d u\|_{L_q(Q_2^+, d\tilde{\mu})},
\end{split}
\]
where in the last estimate, we used $\tilde{\alpha} -1 > q(\alpha-1)$. Therefore
\begin{equation}
                                    \label{0828.eqn}
\|Dw\|_{L_q(Q_2^+, d\tilde{\mu})} \leq N \|x_d^\alpha D u\|_{L_q(Q_2^+, d\tilde{\mu})}.
\end{equation}
Now, due to \eqref{0828.alpha}, we can apply the Sobolev embedding \cite[Lemma 3.1]{DP20} to obtain
\[
\|w\|_{L_{q^*} (Q_2^+, d\tilde{\mu})}\\
 \leq N \Big[ \|w\|_{L_q(Q_2^+, d\tilde{\mu})} + \|D w\|_{L_q(Q_2^+, d\tilde{\mu})} + \|w_t\|_{\bH_q^{-1}(Q_1^+, d\tilde{\mu})} \Big].
\]
Then, by \eqref{0828.eqn}, it follows that
\[
\begin{split}
& \|w\|_{L_{q^*} (Q_2^+, d\tilde{\mu})}\\
& \leq N \Big[ \|x_d^\alpha u\|_{L_q(Q_2^+, d\tilde{\mu})} + \|x_d^\alpha D u\|_{L_q(Q_2^+, d\tilde{\mu})} + \|x_d^\alpha u_t\|_{\bH_q^{-1}(Q_2^+, d\tilde{\mu})} \Big]\\
& = N \|u\|_{\sH^1_q(Q_2^+, x_d^{\alpha q} d\tilde{\mu})}.
\end{split}
\]
This implies \eqref{eq0828} as desired.
\end{proof}
In the time-independent case, we also have the following  embedding result in which the condition \eqref{eq2.06-e} below for $q$ and $q^*$ is optimal.

\begin{lemma}
            \label{imbed-elli}
Let $\alpha, \tilde{\alpha}\in (-\infty, 1)$ and $q \in (1, \infty)$ be fixed numbers such that $\tilde{\alpha} -1>q(\alpha -1)$.  Let $q^*\in (1,\infty)$ satisfy
\begin{equation} \label{eq2.06-e}
                   \frac{1}{q} \le  \frac{1}{d+ \tilde{\alpha}_-}+ \frac{1}{q^*}.
 \end{equation}
Then for any $u\in \sW^1_q(B_2^+, x_{d}^{\alpha q}d\tilde{\mu})$, we have
\begin{equation*}
   \|u\|_{L_{q^*}(B_2^+, x_{d}^{\alpha q^*} d\tilde{\mu})}\le N\|u\|_{\sW^1_{q}(B_2^+, x_{d}^{\alpha q}d\tilde{\mu})},
\end{equation*}
where $N=N(d,\alpha, \tilde{\alpha},q,q^*)>0$ is a constant and $\tilde{\mu}(dx) = x_d^{-\tilde{\alpha}} dx$. The result still holds when $q^*=\infty$ and the inequality in \eqref{eq2.06-e} is strict.
\end{lemma}

\begin{proof} The proof follows as that of Lemma \ref{imbedding-two}. However, instead of applying \cite[Lemma 3.1]{DP20} as in the proof of  Lemma \ref{imbedding-two}, we apply \cite[Remark 3.2 (ii)]{DP20}.
\end{proof}
Now, we give the proof of  Remark \ref{rem2.7}.

\begin{proof}[Proof of Remark \ref{rem2.7}]
The idea of the proof is similar to that of Corollary \ref{main-thrm}. Let  us denote $\tilde{\mu}(dz) = x_d^{-\tilde{\alpha}} dxdt$,  where $\tilde{\alpha} = \frac{\gamma}{p-1} <1$. Then,  by H\"older's inequality it follows that $u \in \sH_{p_0}(Q_2^+, x_d^{\tilde{\alpha}p_0} d\tilde{\mu})$. Using Lemma \ref{imbedding-two}, as in the proof of Corollary \ref{main-thrm} by considering two cases, we can assume without loss of generality that $p$ is not too large and $p^* \leq p_0$ so that
\begin{equation} \label{0831.Aest}
\| x_d^{\tilde{\alpha}} u\|_{L_{p}(Q_2^+, d\tilde{\mu})} < \infty \quad \text{and} \quad \|x_d^{\tilde{\alpha}} Du\|_{L_{p^*}(Q_2^+, d\tilde{\mu})} <\infty.
\end{equation}
Let $w= u \eta$ be as in the proof of Corollary \ref{main-thrm}.  By a direct calculation, we see that $w \in \sH^1_{p_0}(\Omega_0, x_d^{\gamma +\alpha(p_0-p)}dz)$ is a weak solution of
\begin{equation} \label{0826.eqn}
x_d^\alpha (w_t+\lambda w)- D_i\big(x_d^\alpha (a_{ij} D_j w - \widetilde F_i)\big)   =   x_d^\alpha \tilde f  \ \text{in}\ \ (-4,0)\times \bR^d_+
\end{equation}
with the boundary condition $w =0$ on $(-4, 0) \times \partial \bR^d_+$ and the zero initial condition $w(-4,\cdot)=0$,
where
$$
\widetilde F_i=F_i\eta-a_{ij}u D_j\eta,
\quad \tilde f=f\eta+\lambda u\eta +u\eta_t- D_i\eta (a_{ij}D_ju-F_i),
$$
and $\lambda> \lambda_0 R_0^{-2}$ is a fixed constant to be chosen later.

Next, let $q = p/ (p-1)$,  $\tilde{\gamma} = \alpha q - \tilde{\alpha}$,
and $G =(G_1, G_2,\ldots, G_d), g \in C_0^\infty(Q_1^+)$ with
\[
\|G\|_{L_q(Q_1^+,  x_d^{\tilde \gamma}dz)} + \|g\|_{L_q(Q_1^+,  x_d^{\tilde \gamma} dz)} =1.
\]
Because $\gamma \in (p\alpha -1, p-1)$, we have
\[
\alpha q - 1 <\tilde{\gamma}  < q -1.
\]
Let $q_0=p_0/(p_0-1)$ and $\gamma_0=\alpha q_0-\tilde \alpha$. Since $q_0\ge q$, and $\alpha,\tilde\alpha <1$, we also have
\[
\alpha q_0 - 1 <\gamma_0  < q_0 -1.
\]
Moreover, as $G$ and $g$ are compactly supported in $Q_1^+$, $|G| + |g| \in L_{q_0}(Q_1^+, x_d^{\gamma_0}dz)$. Therefore, by Remark \ref{rem2.6}, there exists a weak solution $v \in \sH_{q_0}^1((-4, 0) \times \bR^d_+, x_d^{\gamma_0} dz)$  of
\begin{equation} \label{v0825.eqn}
-x_d^{\alpha}(v_t - \lambda v) - D_i[x_d^\alpha (a_{ji} D_j v -  G_i)] = \sqrt{\lambda} x_d^\alpha g \quad \text{in} \quad (-4, 0) \times \bR^d_+
\end{equation}
with the boundary condition $v =0$ on $(-4, 0) \times \partial \bR^d_+$ and the zero terminal condition $v(0, \cdot) =0$. Since $q\le q_0$, $\tilde \alpha<1$, and $G$ and $g$ are compactly supported, following the proof of Theorem \ref{thm3} \textup{\bf{(ii)}}, we have $v\in \sH^1_{q}((-4,0)\times \bR^d_+, x_d^{\tilde \gamma}dz)$.
Moreover, we also have
\begin{equation} \label{v0828.eqn}
\begin{split}
& \| Dv \|_{L_q((-4, 0) \times \bR^d_+, x_d^{\tilde{\gamma}} dz)} + \sqrt{\lambda} \|v\|_{L_q((-4, 0) \times \bR^d_+, x_d^{\tilde{\gamma}} dz)} \\
& \leq N \big[ \| G\|_{L_q(Q_1^+, x_d^{\tilde \gamma} dz)} +  \|g\|_{L_q(Q_1^+,  x_d^{\tilde \gamma} dz)} \big] = N.
\end{split}
\end{equation}
From this, and the PDE of $v$ in  \eqref{v0825.eqn}, \eqref{v0828.eqn}, and as $\lambda$ sufficiently large, we infer that
\begin{align} \notag
 &\|v\|_{\sH_q^{1}(Q_2^+, x_d^{\tilde{\gamma}} dz)}\notag\\
  & = \| Dv \|_{L_q(Q_2^+, x_d^{\tilde{\gamma}} dz)} +  \|v\|_{L_q(Q_2^+, x_d^{\tilde{\gamma}} dz)}+ \| v_t\|_{\bH^{-1}_q(Q_2^+, x_d^{\tilde{\gamma}} dz)}
  \leq N\sqrt{\lambda}.
  \label{v0829.est}
\end{align}
It can be checked that $\tilde{\alpha}$, $q$, and $q^*$  satisfy the conditions in Lemma \ref{imbedding-two}. Therefore,  it follows from Lemma \ref{imbedding-two} and  \eqref{v0829.est}
\begin{equation} \label{v*0829.est}
\|x_d^\alpha v\|_{L_{q^*}(Q_2^+,  x_d^{-\tilde{\alpha}} dz)} \leq N \|v\|_{\sH_q^{1}(Q_2^+, x_d^{\tilde{\gamma}} dz)} \leq N \sqrt{\lambda}.
\end{equation}
Then, by using $w= u\eta$ as a test function for the equation of $v$, and  $v$ as a test function for the equation of $w$, we have
\[
\int_{Q_1^+} \Big[(x_d^{\alpha} Du) \cdot G + \sqrt{\lambda} (x_d^{\alpha} u) g \Big] dz= \int_{Q_2^+} \Big[(x_d^\alpha Dv) \cdot \tilde{F} + (x_d^\alpha v) \tilde{f} \Big] dz.
\]
By applying H\"{o}lder's inequality and then using \eqref{v0828.eqn} and \eqref{v*0829.est}, we obtain
\[
\begin{split}
&\left| \int_{Q_1^+} \Big[(x_d^{\alpha} Du) \cdot G + \sqrt{\lambda} (x_d^{\alpha} u) g \Big] dz \right | \\
& \leq \|Dv\|_{L_q(Q_2^+, x_d^{\tilde{\gamma}} dz)}\| \tilde{F}\|_{L_p(Q_2^+, x_d^\gamma dz)} + \|x_d^{\alpha} v\|_{L_{q^*}(Q_2^+, x_d^{-\tilde{\alpha}} dz)} \| \tilde{f}\|_{L_{p^*}(Q_2^+, x_d^{\tilde{\alpha}(p^*-1)} dz)} \\
& \leq N\| \tilde{F}\|_{L_p(Q_2^+, x_d^\gamma dz)} + N \sqrt{\lambda}  \| \tilde{f}\|_{L_{p^*}(Q_2^+, x_d^{\tilde{\alpha}(p^*-1)} dz)}.
\end{split}
\]
Observe that
\[
\| \tilde{F}\|_{L_p(Q_2^+, x_d^\gamma dz)} \leq N\Big[ \|F\|_{L_p(Q_2^+, x_d^\gamma dz)} + \| u\|_{L_p(Q_2^+, x_d^\gamma dz)}\Big]
\]
and
\[
\begin{split}
\| \tilde{f}\|_{L_{p^*}(Q_2^+, x_d^{\tilde{\alpha}(p^*-1)}dz)} & \leq N \Big[\|f\|_{L_{p^*}(Q_2^+, x_d^{\tilde{\alpha}(p^*-1)}dz)} + \|F\|_{L_{p^*}(Q_2^+, x_d^{\tilde{\alpha}(p^*-1)}dz)} \\
& \quad + \lambda \|u\|_{L_{p^*}(Q_2^+, x_d^{\tilde{\alpha}(p^*-1)}dz)}  + \|Du\|_{L_{p^*}(Q_2^+, x_d^{\tilde{\alpha}(p^*-1)}dz)}\Big].
\end{split}
\]
As $p^* \leq p$ and $\gamma < p-1$, by H\"{o}lder's inequality, we obtain
\[
\begin{split}
 \|F\|_{L_{p^*}(Q_2^+, x_d^{\tilde{\alpha}(p^*-1)}dz)}
& \leq  \|F\|_{L_{p}(Q_2^+, x_d^{\gamma}dz)} \left( \int_{Q_2^+} x_d^{-\frac{\gamma}{p-1}} dz \right)^{ 1/p^*-1/p}\\
& \leq N  \|F\|_{L_{p}(Q_2^+, x_d^{\gamma}dz)}.
\end{split}
\]
Therefore,
\[
\begin{split}
&\left| \int_{Q_1^+} \Big[(x_d^{ \alpha} Du) \cdot G
+ \sqrt{\lambda} (x_d^{ \alpha} u) g \Big] dz \right | \\
& \leq N (\sqrt{\lambda} +1)\|F\|_{L_{p}(Q_2^+, x_d^{\gamma}dz)} + N \sqrt{\lambda} \| f\|_{L_{p^*}(Q_2^+, x_d^{\tilde{\alpha}(p^*-1)}dz)}+ N \| u\|_{L_p(Q_2^+, x_d^\gamma dz)}  \\
& +  N\lambda^{3/2} \|u\|_{L_{p^*}(Q_2^+, x_d^{\tilde{\alpha}(p^*-1)}dz)} + N\sqrt{\lambda} \|Du\|_{L_{p^*}(Q_2^+, x_d^{\tilde{\alpha}(p^*-1)}dz)}.
\end{split}
\]
Since $G$ and $g$ are arbitrary  and $\gamma/p+\tilde\gamma/q=\alpha$, we infer that
\[
\begin{split}
& \| Du\|_{L_p(Q_1^+, x_d^\gamma dz)} + \sqrt{\lambda}\|u\|_{L_p(Q_1^+, x_d^\gamma dz)}\\
& \leq N (\sqrt{\lambda} +1)\|F\|_{L_{p}(Q_2^+, x_d^{\gamma}dz)} + N \sqrt{\lambda} \| f\|_{L_{p^*}(Q_2^+, x_d^{\tilde{\alpha}(p^*-1)}dz)}+ N \| u\|_{L_p(Q_2^+, x_d^\gamma dz)}  \\
& +  N\lambda^{3/2} \|u\|_{L_{p^*}(Q_2^+, x_d^{\tilde{\alpha}(p^*-1)}dz)} + N\sqrt{\lambda} \|Du\|_{L_{p^*}(Q_2^+, x_d^{\tilde{\alpha}(p^*-1)}dz)},
\end{split}
\]
which is equivalent to
\begin{equation} \label{u0830.est}
\begin{split}
& \|x_d^{\tilde\alpha} Du\|_{L_p(Q_1^+, d\tilde \mu)}
+ \sqrt{\lambda}\|x_d^{\tilde\alpha} u\|_{L_p(Q_1^+, d\tilde \mu)}\\
& \leq N (\sqrt{\lambda} +1)\|x_d^{\tilde\alpha} F\|_{L_{p}(Q_2^+, d\tilde \mu)}
+ N \sqrt{\lambda} \| x_d^{\tilde\alpha} f\|_{L_{p^*}(Q_2^+, d\tilde \mu)} \\
& + N \|x_d^{\tilde\alpha} u\|_{L_p(Q_2^+, d\tilde \mu)} + N\lambda^{3/2} \|x_d^{\tilde\alpha} u\|_{L_{p^*}(Q_2^+, d\tilde \mu)}
+ N\sqrt{\lambda} \|x_d^{\tilde\alpha} Du\|_{L_{p^*}(Q_2^+, d\tilde \mu)},
\end{split}
\end{equation}
where $d\tilde\mu=x_d^{-\tilde\alpha}dz$.  Note that by \eqref{0831.Aest} and the assumptions in the remark, all the terms on the right-hand side of \eqref{u0830.est} are finite. Then as $p^*<p$, we conclude \eqref{remark-estz} from \eqref{u0830.est} by using H\"older's inequality and a standard iteration argument with the underlying measure $d\tilde \mu$ and for a sufficiently large $\lambda$.  See, for example, \cite[pp. 80--82]{Giaq}. The remark is proved in this case.
\end{proof}

%====
\begin{proof}[Proof of Corollary \ref{ell-local-th}] As discussed in Remark \ref{rem2.6}, it follows from Theorem \ref{ell-thm} that there exists unique weak solution $u \in \sW^1_p(\bR^d_+, x_d^\gamma dx)$ for \eqref{ell-eqn}. From this and Lemma \ref{imbed-elli}, we can follow the proof of Remark \ref{rem2.7} to obtain the assertion in Corollary \ref{ell-local-th}.
\end{proof}

\end{document}